\documentclass[11pt]{article}

\usepackage[margin=2cm]{geometry}
\usepackage{amssymb,amsmath,amsthm}
\usepackage{latexsym}
\usepackage{graphicx}
\usepackage{verbatim}
\usepackage[all]{xy}

\theoremstyle{plain}
\newtheorem{theorem}{Theorem}[section]
\newtheorem{lemma}[theorem]{Lemma}
\newtheorem{corollary}[theorem]{Corollary}
\newtheorem{proposition}[theorem]{Proposition}

\theoremstyle{definition}
\newtheorem{definition}[theorem]{Definition}
\newtheorem{remark}[theorem]{Remark}
\newtheorem{example}[theorem]{Example}

\newcommand{\lp}{\left(} % left parentheses (
\newcommand{\rp}{\right)} % right parentheses )

\newcommand{\dash}{\mathchar `\-}

\newcommand{\tnorm}[1]{
  \left\vert\kern-0.9pt\left\vert\kern-0.9pt\left\vert #1
    \right\vert\kern-0.9pt\right\vert\kern-0.9pt\right\vert}

\newcommand{\dd}{\,{\rm d}}
\newcommand{\inj}{{\rm inj}}

\newcommand{\support}{\mathop{\rm supp\,}}

\newcommand {\AB}{{}_A\mathcal{B}}

\newcommand {\Amod}{A\dash{\bf mod}}

\newcommand {\Lmod}{L^{1}(G)\dash{\bf mod}}
\newcommand {\lmod}{\ell^{\, 1}(G)\dash{\bf mod}}
\newcommand {\lSmod}{\ell^{\, 1}(S)\dash{\bf mod}}

\newcommand {\modA}{{\bf mod}\dash A}
\newcommand {\modL}{{\bf mod}\dash L^{1}(G)}

\newcommand {\AmodA}{A\dash{\bf mod}\dash A}
\newcommand {\LmodL}{L^1(G)\dash{\bf mod}\dash L^1(G)}

\newcommand{\enproof}{\hspace*{\stretch{1}}$\square$}%\qedsymbol}  % end of proof box

\newcommand{\norm}[1]{\|#1\|}  % norm
\newcommand{\flexiblenorm}[1]{\left\|#1\right\|}
\newcommand{\multibound}{\operatorname{mb}}
\newcommand{\Lspace}{\mathit{L}}
\newcommand{\Lone}{\Lspace^1}
\newcommand{\Linfty}{\Lspace^\infty}
\newcommand{\lspace}{\ell}
\newcommand{\lone}{\lspace^{\,1}}

\newcommand{\linfty}{\lspace^{\,\infty}}
\newcommand{\co}{\mathit{c}_{\,0}}
\newcommand{\coo}{\mathit{c}_{00}}
\newcommand{\C}{\mathit{C}}
\newcommand{\set}[1]{\left\{#1\right\}}  % set

\newcommand{\complexs}{\mathbb{C}}    % complex numbers
\newcommand{\reals}{\mathbb{R}}    % real numbers
\newcommand{\naturals}{\mathbb{N}}    % natural numbers

\newcommand{\tuple}[1]{\boldsymbol{#1}}
\newcommand{\E}{\mathcal{E}}
\newcommand{\V}{\mathcal{V}}

\newcommand{\abs}[1]{\left|#1\right|}
\newcommand{\cardinal}[1]{\left|#1\right|}

\newcommand{\dual}[1]{#1^\prime}
\newcommand{\bidual}[1]{#1^{\prime\prime}}

\newcommand{\duality}[2]{\left\langle #1,#2 \right\rangle}

\newcommand{\operators}{\mathcal{B}}
\newcommand{\compactoperators}{\mathcal{K}}
\newcommand{\finiterankoperators}{\mathcal{F}}
\newcommand{\nuclearoperators}{\mathcal{N}}
\newcommand{\mboperators}{\mathcal{M}}
\newcommand{\id}{I}

\newcommand{\projectivetensor}{\,\widehat{\otimes}\,}

\newcommand{\injectivetensor}{\!\stackrel{\textbf{\tiny \reflectbox{\rotatebox[origin=c]{90}{$\!\!\!\mathbf{\langle}\!$}}}}{\otimes}\!}

\title{Multi-norms and the injectivity of $\Lspace^p(G)$}

\author{H.\ G.\ Dales,  M.\ Daws, H.\ L.\ Pham, and P.\ Ramsden}

\begin{document}
\maketitle

\begin{abstract}
Let $G$ be a locally compact group, and take $p\in(1,\infty)$. We prove that the Banach left $\Lone(G)$-module $\Lspace^p(G)$ is injective (if and) only if the group $G$ is  amenable. Our proof uses the notion of multi-norms. We also develop the theory of multi-normed spaces. 

(2010) Subject classification: 46H25, 43A20.
\end{abstract}

\section{Introduction}

\noindent Let $G$ be a locally compact group, and let $\Lone(G)$ be the group algebra of $G$. In \cite{DP}, H. G. Dales and M. E. Polyakov investigated
 when various canonical modules over $\Lone(G)$ have certain well-known homological properties. For example, it was proved in \cite[Theorem 4.9]{DP} that $\Lone(G)$
 is injective in $\Lmod$, the category of Banach left $L^1(G)$-modules,  if and only if $G$ is discrete and amenable, and in \cite[Theorem 2.4]{DP} that $\Linfty(G)$
 is injective in $\Lmod$ for every locally compact group $G$.

One of the more difficult questions that they considered seems to have been to characterize the locally compact groups $G$ such that the Banach left $\Lone(G)$-module 
$\Lspace^{p}(G)$ is injective (for $1<p< \infty$).  By Johnson's famous theorem \cite{Johnson72}, the Banach algebra $\Lone(G)$ is amenable if and only
 if $G$ is an amenable group.  Since $\Lspace^p(G)$ is a dual Banach $\Lone(G)$-module, it follows from  
 \cite[VII.2.2]{Helemskii86} and from \cite[\S 5.3]{Runde} that $\Lspace^{p}(G)$ is an injective Banach left $\Lone(G)$-module whenever $G$ is amenable as a locally compact group; 
the converse has been an open problem for a long time. In \cite{DP}, the authors obtained a partial converse to  
 this theorem in the case where $G$ is discrete. Indeed, they showed the following \cite[Theorem 5.12]{DP}. Let $G$ be a group, and suppose that $\lspace^{\, p}(G)$ is an injective Banach left $\lone(G)$-module for some $p \in (1,\infty)$. Then $G$ must be `pseudo-amenable', a property very close to amenability. (In fact, no example of a group that is pseudo-amenable, but not amenable, is known.)  

In this paper, we shall define another generalized notion of amenability, called {\it left $(p,q)$-amenability} of $G$, for any $p,q$ such that $1\leq p\leq q<\infty$ and for any locally compact group $G$. We shall show the following for each $p,q$ with $1<p\le q<\infty$:
\[
\Lspace^{p}(G) \text{~is injective} \iff G \textrm{~is left }(p,q)\dash\text{amenable}  
\iff G \text{~is amenable}\,.
\]
In particular, we resolve positively the above open problem. As a consequence, we shall also determine when the module $\Lspace^p(G)$ is flat. 

In the final section \S \ref{Semigroup algebras}, we shall give some similar results for the modules $\lspace^{\,p}(S)$, regarded as Banach left $\lone(S)$-modules, for a cancellative semigroup $S$.

Our definition of left $(p,q)$-amenability is framed in the language of `multi-norms'. The theory of multi-norms was developed by Dales and Polyakov in an attempt to resolve 
the above-mentioned problem. However, this theory has developed a life of its own; it is expounded at some length in \cite{DP08}, where many examples are given, and the 
connection with various known `summing norms' and `summing constants' is explained. We shall give a presentation of multi-norms and their duals in terms of certain tensor
 norms in \S \ref{Multi-normed spaces as tensor norms} and \S \ref{Dual multi-normed spaces}.
\medskip

After this paper was submitted for publication, we received  the preprint \cite{Racher2} from Professor Gerhard Racher (Salzburg).  This preprint states the following
  more general version  of Theorem 9.6. `Let $G$ be a locally compact group.  Suppose that there exists a non-zero, injective  Banach left $L^1(G)$-module that  is reflexive as a Banach space. Then $G$ is amenable as a locally compact group.' We thank Professor  Racher for sending us this preprint.

\medskip

\noindent\textbf{Acknowledgments:}
We are grateful to the referee for some valuable comments.
We would like to acknowledge the financial support of EPSRC under grant EP/H019405/1 awarded to H. G. Dales, of the London Mathematical Society under grant Scheme 2, ref. 2901 awarded  to M. Daws, and of the Marsden Fund (the Royal Society of New Zealand) awarded to H. L. Pham.

\section{Background and notation}

\noindent In this section, we shall recall various notations that we shall use, and give the definitions and some properties of multi-norms and multi-bounded operators. We shall also recall the definitions of the group algebra $\Lone(G)$ and the Banach left $\Lone(G)$-modules $\Lspace^p(G)$ for a locally compact group $G$ and $p\ge 1$.

\subsection{Banach spaces}

\noindent For $n \in \naturals=\set{1,2,\ldots}$, we set $\naturals_n=\set{1,\ldots,n}$. The cardinality of a set $S$ is $\cardinal{S}$, and the characteristic function of a subset $T$ of  $S$ is denoted by $\chi_{T}$; we set $\delta_{s}=\chi_{\set{s}}\,\; (s \in S)$. 
The conjugate to a number $p\ge 1$ is sometimes denoted by $p'$, so that $1/p+1/p'=1$.

Let $E$ be a linear space. The identity operator on $E$ is $\id_E$. For each $k \in \naturals$, we denote by $E^k$ the linear space direct product of $k$ copies of $E$.   
Let $F$ be another linear space, and let $T:E\rightarrow F$ be a linear mapping. Then we define the linear map $T^{(k)}:E^k\rightarrow F^k$, the $k^{\rm th}$-\emph{amplification} of $T$,  by
\[
T^{(k)}(x_{1},\ldots,x_k)=(Tx_1,\ldots,Tx_k)\quad (x_1,\ldots, x_k \in E)\,.
\]

Let $E$ be a normed space. Then the closed unit ball of $E$ is denoted by $E_{[1]}$. We denote the dual space of $E$ by $\dual{E}$; the action of $\lambda \in \dual{E}$ on an element $x \in E$ is written as $\duality{x}{\lambda}$. 

Let $E$ and $F$ be normed spaces. Then $\operators(E, F)$ is the normed space of all bounded linear operators from $E$ to $F$ with the operator norm; the dual of an operator $T\in \operators(E, F)$ is denoted by $\dual{T}\in \operators(\dual{F}, \dual{E})$.  The subspaces of $\operators(E,F)$ consisting of the finite-rank  
and of the compact operators are denoted by $\finiterankoperators(E,F)$ 
and $\compactoperators(E,F)$, respectively; we write $\finiterankoperators(E)$  
and $\compactoperators(E)$ for $\finiterankoperators(E,E)$  
and $\compactoperators(E,E)$, respectively. For $\lambda \in \dual{E}$ and $y \in F$, we define the rank-one operator $\lambda\otimes y\in \operators(E, F)$ by
\begin{align}\label{tensor form of finite-rank}
(\lambda\otimes y)(x)=\duality{x}{\lambda} y\quad (x \in E)\,;
\end{align}
in this way, we identify the tensor product $\dual{E}\otimes F$ with $\finiterankoperators(E,F)$.

Let $E$ be a normed space, and take $n\in\naturals$. Following the notation of \cite{DP08} and \cite{Jameson}, we define the \emph{weak $p$-summing norm}
 (for $1\leq p<\infty$) on $E^n$ by
\[
\mu_{p,n}(\tuple{x})=\sup \set{\left(\sum_{i=1}^{n}\abs{\duality{x_{i}}{\lambda}}^{p}\right)^{1/p}: \lambda \in \dual{E}_{[1]}}\index{$\mu_{p,n}(\tuple{x})$}\,,
\]
where $\tuple{x}=(x_1,\ldots,x_n) \in E^n$. See also \cite[p. 32]{DJT} and \cite[p. 134]{Ryan}.  
Notice that, by the weak$^*$-density of $E_{[1]}$ in $\bidual{E}_{[1]}$, the weak $p$-summing norm on $(\dual{E})^n$ can also be computed as
\begin{align}\label{weak p-summing norm on dual space}
	\mu_{p,n}(\tuple{\lambda})=\sup\set{\left(\sum_{i=1}^n\abs{\duality{x}{\lambda_i}}^p\right)^{1/p}\colon x\in E_{[1]}}\,,
\end{align}
where $\tuple{\lambda}=(\lambda_1,\ldots,\lambda_n)\in(\dual{E})^n$.

Let $K$ be a non-empty, locally compact space; our convention is that locally compact spaces are Hausdorff. Then $\C_0(K)$ is the
 Banach space of complex-valued, continuous functions which vanish at infinity on $K$, equipped with the uniform norm $\abs{\,\cdot\,}_{K}$, given by
\[
\abs{f}_{K}=\sup\set{ \abs{f(x)}: x \in K}\quad (f \in \C_0(K))\,.
\]
Further $\C_{00}(K)$ is the subspace of $\C_0(K)$ of functions with compact support.

Let $(\Omega,\mu)$ be a measure space, and take $p\ge 1$. Then $\Lspace^p(\Omega)=\Lspace^p(\Omega, \mu)$ is the Banach space of
 (equivalence classes of) complex-valued, $p\,$-integrable functions on $\Omega$, equipped with the norm $\norm{\cdot}_p$, given by
\[
\norm{f}_p=\lp \int_\Omega \abs{f}^p \dd \mu\rp^{1/p}\quad (f \in \Lspace^p(\Omega))\,.
\]
Of course, the dual of $\Lspace^p(\Omega)$ is identified with $\Lspace^{p'}(\Omega)$ when $p>1$.

Let $\co$ and $\lspace^{\,p}$ be the usual Banach spaces.  We write $(\delta_n)_{n=1}^\infty$ for the
standard basis for $\co$ and $\lspace^{\,p}$.  For $n\in\naturals$, we write $\linfty_n$
for $\complexs^n$ with the supremum norm, and we regard each $\linfty_n$ as a subspace
of $\co$, and hence regard $(\delta_i)_{i=1}^n$ as a basis for $\linfty_n$.

\subsection{Banach homology}

\noindent For the homological background to our work,  we refer the reader to the standard reference \cite{Helemskii86}; for a clear account of all that we require, see \cite[Chapter 5]{Runde}. We briefly sketch what we shall need. 

Let $A$ be a Banach algebra, and let $E$ be a Banach space that is a left $A$-module for the map
\[
	(a,x)\mapsto a\cdot x\,,\quad A\times E\to E\,.
\]
Then $E$ is a \emph{Banach left $A$-module} if there is a constant $C>0$ such that 
\[
	\norm{a\cdot x}\le C \norm{a}\norm{x}\quad(a\in A,\ x\in E)\,;
\]
we denote by $\Amod$ the category of Banach left $A$-modules. Similarly, $\modA$ and $\AmodA$ are the categories of Banach right $A$-modules and Banach $A$-bimodules, respectively. (See \cite{HGD}, \cite{Helemskii86}, \cite{Johnson72}, and \cite{Runde}, for example.)

Let $E \in \Amod$. Then $E$ is {\it essential\/} if the linear span of the elements $a\,\cdot\,x$ for $a\in A$ and $x \in E$ is dense in $E$. 
 In the case where  $A$ has a bounded left approximate identity, this implies \cite[Corollary 2.9.26]{HGD}  that $E$ is {\it neo-unital\/}, in the sense
 that each element in $E$ has the form  $a\,\cdot\,x$ for some $a\in A$ and $x \in E$.

Let $E\in \Amod$. Then the dual action of $A$ on $\dual{E}$ is defined by
\[
	\duality{x}{\lambda\cdot a}=\duality{a\cdot x}{\lambda}\quad(a\in A,\ x\in E,\ \lambda\in \dual{E})\,,
\]
and then $\dual{E}\in\modA$ is the \emph{dual module to}  $E$. Similarly, $\dual{E}\in\Amod$ when $E\in\modA$.

For spaces $E, F \in \Amod$, the Banach space of bounded $A$-module morphisms from $E$ to $F$ is denoted by $\AB(E, F)$. A monomorphism $T \in \AB(E, F)$ is said to be \emph{admissible} if there exists $S \in {\cal B}(F,E)$ such that \mbox{$S\,\circ\,T =\id_E$}, and $T$ is a {\it coretraction} if there exists $S \in \AB(F, E)$ such that $S\circ T=\id_{E}$.

\begin{definition} \label{1.1b}
Let $A$ be a Banach algebra, and let $J\in \Amod$.  Then $J$ is {\it injective} if, for each $E,F\in \Amod$, for each admissible monomorphism $T\in \AB(E,F)$, and for each $S\in \AB(E,J)$, there exists $R\in \AB(F,J)$ such that $R\,\circ\,T=S$.
\end{definition}

Let $A$ be a Banach algebra, and let $E$ be a Banach space. Then ${\cal B}(A,E) \in \Amod$ when we define the module operation by the formula
\[
(a\cdot T)(b)=T(ba)\quad (a, b\in A,\,T\in \operators(A, E))\,.
\]
Now suppose that $E \in \Amod$. Then we define the \emph{canonical embedding} $\Pi:E\rightarrow {\cal B}(A,E)$ by the formula
\[
\Pi(x)(a)=a\cdot x\quad (a\in A,\,x\in E)\,,
\]
so that $\Pi\in\AB(E,\operators(A,E))$.  The mapping $\Pi$ is indeed an embedding if $E$ has the property that $x=0$ whenever $x\in E$ and $a\,\cdot\,x = 0$ 
for all $a \in A$; i.e., $\{x\in E : A\,\cdot\,x= \{0\}\}= \{0\}$.  This property holds whenever $A$ has a bounded left approximate identity and $E$ is essential.

For background, we note the following characterization of injective modules {\cite[Proposition 1.7]{DP}}; we shall use related ideas.  

\begin{proposition}\label{2.2}
Let $A$ be a Banach algebra, and let $E\in \Amod$ have the property  that $\{x\in E : A\,\cdot\,x= \{0\}\}= \{0\}$.
Then the module $E$ is injective if and only if the morphism $\Pi\in{\AB}(E, \operators(A, E))$ is a coretraction in $\Amod$. \enproof 
\end{proposition}

We shall take the following as our definition of a flat module; a different, more intrinsic, definition is given in \cite[VI.1.1]{Helemskii86} and \cite[Definition 5.3.3]{Runde},
 and the equivalence of the two definitions is shown in \cite[VII.1.14]{Helemskii86} and \cite[Theorem 5.3.8]{Runde}.

\begin{definition} \label{definition of flat}
Let $A$ be a Banach algebra, and let $E\in \Amod$.  Then $E$ is {\it flat} if the dual module $\dual{E}$ is injective in $\modA$.
\end{definition}

We note that every projective module in $\Amod$ is flat \cite[Examples 5.3.9(b)]{Runde}. We shall use the following basic result of Helemskii \cite{He1}: 
see  \cite[VII.2.29]{Helemskii86} and \cite[Theorem 5.3.8 and Example 5.3.9(a)]{Runde}. The notion of an amenable Banach algebra originates with Johnson 
\cite{Johnson72}; see \cite[\S 2.8]{HGD}, \cite{Helemskii86}, and \cite{Runde}.

\begin{theorem}
Let $A$ be an amenable Banach algebra. Then every dual module in $\Amod$ or $\modA$ is injective; equivalently, every module in $\Amod$ or $\modA$ is flat. \enproof 
\end{theorem}

\subsection{$\Lspace^p$ modules over group algebras} 

\label{modules over group algebras}

\noindent Let $G$ be a locally compact group with left Haar measure $m$ and modular function $\Delta$, and set $\Lone(G)=\Lone(G,m)$; see \cite[\S 3.3]{HGD}. For $f \in \Lone(G)$ 
and $s \in G$, we define $s\cdot f \in \Lone(G)$ by 
\[
	(s\cdot f)(t)=f(s^{-1}t)\quad(t \in G)\,,
\]
so defining an action of $G$ on the space $\Lone(G)$. We can extend this action by duality to the space $\dual{\Linfty(G)}=\bidual{\Lone(G)}$. An element $\Lambda \in \dual{\Linfty(G)}$ is a {\it mean} on $\Linfty(G)$ if 
\[
	\duality{1}{\Lambda}=\norm{\Lambda}=1\,,
\]
and $\Lambda$ is {\it left-invariant} if $\set{s\cdot \Lambda:s \in G}=\set{\Lambda}$. The group $G$ is {\it amenable} if there exists a left-invariant mean on $\Linfty(G)$. There are many different characterizations of the amenability of $G$; see \cite{Paterson}, for example, for a full account.

Let $G$ be a locally compact group. We now consider $\Lone(G)$ as a Banach algebra equipped with the {\it convolution product} $\star$ given by
\begin{equation}\label{convolution}
(f \star g)(s)=\int_{G} f(t)g(t^{-1}s) \dd m(t)\quad (s \in G)\,,
\end{equation}
where $f, g \in \Lone(G)$ and the integral is defined for almost all $s \in G$. It is standard that $(\Lone(G), \,\star\,)$  has a bounded approximate identity.
It is a very famous theorem of Johnson \cite{Johnson72} that the algebra $\Lone(G)$ is amenable as a Banach algebra if and only if the locally compact group $G$ is amenable; see also \cite[Theorem 5.6.42]{HGD}.

We denote by $\varphi_G$ the {\it augmentation character} on $G$, given by
\[
\varphi_G(f)=\int_G f(t) \dd m(t)\quad (f \in \Lone(G))\,.
\]

Let $p\in[1,\infty)$, and set $E= \Lspace^p(G)=\Lspace^{\, p}(G, m)$. Take $f \in \Lone(G)$ and $g \in \Lspace^p(G)$. Then again we can define $f \star g$ on $G$ via \eqref{convolution}, and in this case we have $f\star g \in \Lspace^p(G)$. With this multiplication, $\Lspace^p(G)$ has the structure of a Banach left $\Lone(G)$-module;  indeed, we have  $\Lspace^p(G)\in\Lmod$ \cite[Theorem 3.3.19]{HGD}.  The module $E$ is essential, and so Proposition \ref{2.2} applies.

In fact, the spaces $\Lspace^p(G)$ are Banach $\Lone(G)$-bimodules, where the right module action of $\Lone(G)$ on $\Lspace^p(G)$ is defined as 
\[
	(g\star f)(s)=\int g(st^{-1})f(t)\Delta^{1/p}(t^{-1}) \dd m(t)\,\quad(s\in G)
\]
for $f \in \Lone(G)$ and $g \in \Lspace^{p}(G)$. This formula in the case where $p=1$ gives the same right action as the  convolution product on $\Lone(G)$.

We shall use the notation $\,\cdot\,$ for the module products on $\Lspace^{p'}(G)$ considered as the dual module of ${(\Lspace^p(G),\star)}$; the formulae for these products are given in \cite[\S 3.3]{HGD}. These dual module actions are similar to, but different from, the standard actions $\star$ defined above.

\subsection{Multi-normed spaces}

\noindent The following definition is due to Dales and Polyakov. For a full account of the theory of multi-normed spaces, see \cite{DP08}.

\begin{definition}
Let $(E, \norm{\cdot})$ be a normed space, and let $(\norm{\cdot}_n: n \in \naturals)$ be a sequence such that $\norm{\cdot}_n$ is a norm on $E^n$ for each $n \in \naturals$, with $\norm{\cdot}_1=\norm{\cdot}$ on $E$. Then the sequence $(\norm{\cdot}_n: n \in \naturals)$ is a \emph{multi-norm} if the following axioms hold (where in each case the axiom is required to hold for all $n\geq 2$ and all $x_1,\ldots, x_n \in E$):\medskip

\noindent(A1) $\norm{(x_{\sigma(1)},\ldots,x_{\sigma(n)})}_{n}=\norm{(x_{1},\ldots,x_{n})}_{n}$ for each permutation $\sigma$ of $\naturals_n$;\medskip

\noindent(A2) $\norm{(\alpha_1x_{1},\ldots,\alpha_nx_{n})}_{n}\leq \max_{i \in \naturals_n}\abs{\alpha_i}\norm{(x_{1},\ldots,x_{n})}_{n} \quad (\alpha_1,\ldots,\alpha_n \in \complexs)$\,;\medskip

\noindent(A3) $\norm{(x_1,\ldots,x_{n-1},0)}_{n}=\norm{(x_1,\ldots,x_{n-1})}_{n-1}$\,;\medskip

\noindent(A4) $\norm{(x_1,\ldots,x_{n-2},x_{n-1},x_{n-1})}_{n}=\norm{(x_1,\ldots,x_{n-2},x_{n-1})}_{n-1}$\,.\medskip

\noindent The normed space $E$ equipped with a multi-norm is a \emph{multi-normed space}, denoted  in full by $((E^n,\norm{\cdot}_n):\ n\in\naturals)$. We say that such a multi-norm is \emph{based on} $E$.
\end{definition}

Suppose that in the above definition we replace axiom (A4) by the following axiom:\medskip

\noindent(B4) $\norm{(x_1,\ldots,x_{n-2},x_{n-1},x_{n-1})}_{n}=\norm{(x_1,\ldots,x_{n-2},2x_{n-1})}_{n-1}$.\medskip

\noindent Then we obtain the definition of a {\it dual multi-norm} and of a {\it dual multi-normed space}. (A yet more general concept,
 that of sequences $(\norm{\cdot}_n:\ n\in\naturals)$ satisfying just (A1)--(A3), is mentioned in \cite[$\S2.2.1$]{DP08}.)

Let $((E^n,\norm{\cdot}_n):\ n\in\naturals)$ be a multi-normed or dual multi-normed space.  For each $n\in\naturals$,  the dual of the space 
$(E^n,\norm{\cdot}_n)$ can be isomorphically identified with the Banach space $(E')^n$, as explained in \cite[$\S 1.2.4$]{DP08}, 
and in this way we regard $(E')^n$ as a Banach space. The weak$^*$ topology from this duality is the product of the weak$^*$ topologies 
given by the duality of $E$ and $E'$. 

The following results are noted in \cite[Chapter 2]{DP08}. First, the axioms (A1)--(A4) are independent \cite[$\S 2.1.3$]{DP08}. 
 Second, in the case where $(\norm{\cdot}_n: n \in \naturals)$ satisfies (A1)--(A3), we have
\[
	\max_{i\in\naturals_n}\norm{x_i}\le \norm{(x_1,\ldots,x_n)}_n\le\sum_{i=1}^n\norm{x_i}\quad(x_1,\ldots,x_n\in E)
\]
for each $n\in\naturals$, and so $\norm{\cdot}_n$ defines the same topology on $E^n$ as the product topology \cite[Lemma 2.11]{DP08}. 
Third, if $(\norm{\cdot}_n: n \in \naturals)$ is a multi-norm or a dual multi-norm based on $E$, and $\dual{\norm{\cdot}_n}$ is the dual 
norm to $\norm{\cdot}_n$ for each $n\in\naturals$, then $(\dual{\norm{\cdot}}_n: n \in \naturals)$ is a dual multi-norm or multi-norm, respectively, 
based on $\dual{E}$ \cite[$\S 2.3.2$]{DP08}.  This latter result implies that the sequence of second duals of a multi-norm
 $(\norm{\cdot}_n: n \in \naturals)$ is a multi-norm based on $\bidual{E}$.

The family $\E_E$ of all multi-norms based on a normed space $E$ is a Dedekind complete lattice  with respect to the  ordering $\le$, 
where $(\norm{\cdot}^1_n: n \in \naturals)\le (\norm{\cdot}^2_n: n \in \naturals)$ if 
\[
	\norm{\tuple{x}}^1_n\le \norm{\tuple{x}}^2_n\quad(\tuple{x}\in E^n,\ n\in\naturals)
\] 
\cite[Proposition 3.10]{DP08}. The minimum element of the lattice  $(\E_E,\le)$ 
is the \emph{minimum multi-norm} $(\norm{\cdot}^{\min}_n: n \in \naturals)$, and the formula for $\norm{\cdot}^{\min}_n$ is 
\[
	\norm{(x_1,\ldots,x_n)}^{\min}_n=\max_{i\in\naturals_n}\norm{x_i}\quad(x_1,\ldots, x_n\in E)
\]
for each $n\in\naturals$, as in \cite[Definition 3.2]{DP08}.

For each normed space $E$, there is a unique maximum element in the lattice $(\E_E,\le)$; this is the \emph{maximum multi-norm} $(\norm{\cdot}_n^{\max}: n \in \naturals)$. By \cite[Theorem 3.33]{DP08}, for each $\tuple{x}=(x_1,\ldots,x_n) \in E^n$ and each $n \in \naturals$, we have
\begin{equation}\label{max-mn}
\norm{\tuple{x}}^{\max}_n=\sup \set{ \abs{\sum_{i=1}^n\duality{ x_i}{\lambda_i}}: \lambda_1,\ldots, \lambda_n \in \dual{E},\, \mu_{1,n}(\lambda_1,\ldots, \lambda_n)\leq 1}\,.
\end{equation}
Further, $\mu_{1,n}$ on $(E')^n$ is the dual norm to the norm $\norm{\cdot}^{\max}_n$ on $E^n$.

A multi-norm $(\norm{\cdot}^2_n: n \in \naturals)$ in $\E_E$ \emph{dominates} a multi-norm $(\norm{\cdot}^1_n: n \in \naturals)$ in $\E_E$ if there is a constant $C>0$ such that $\norm{\tuple{x}}^1_n\le C\norm{\tuple{x}}^2_n$\,\; ($\tuple{x}\in E^n,\ n\in\naturals$); two multi-norms are \emph{equivalent} if each dominates the other. 

The following is {\cite[Definition 6.4]{DP08}} (where  $c_B$ is used instead of our $\multibound(B)$ to denote the multi-bound of a set $B$).

\begin{definition}
Let $((E^n,\norm{\cdot}_n):\ n\in\naturals)$ be a multi-normed space. A subset $B\subset E$ is {\it multi-bounded} if
\[
\multibound(B):=\sup \set{ \norm{(x_1,\ldots, x_n)}_n: x_1,\ldots, x_n \in B,\, n \in \naturals}<\infty\,.
\]
The constant $\multibound(B)$ is the \emph{multi-bound} of $B$.
\end{definition}

The following easy remark is \cite[Proposition 6.5(ii)]{DP08}.

\begin{lemma}\label{absolutely convex hull and multi-bounded}
Let $E$ be a multi-normed space. Then the absolutely convex hull of a multi-bounded set is multi-bounded, with the same multi-bound.\enproof
\end{lemma}

\begin{definition}
Let $((E^n,\norm{\cdot}_n):\ n\in\naturals)$ and $((F^n,\norm{\cdot}_n):\ n\in\naturals)$ be multi-normed spaces, and let $T \in \mathcal{B}(E, F)$. Then $T$ is \emph{multi-bounded} if 
\[
\norm{T}_{mb}:=\sup_{k \in \naturals} \norm{T^{(k)}}< \infty\,.\index{$T_{mb}$}
\]
We set 
\[
	\mboperators(E, F)=\set{T \in \operators(E, F): \norm{T}_{mb}<\infty}\,,
\]
so that $\mboperators(E, F)$ is the space of {\it multi-bounded operators}.\index{$M(E, F)$}
\end{definition}

Here, $\norm{T^{(k)}}$ is calculated by regarding $T^{(k)}$ as a bounded linear map from $(E^k,\norm{\cdot}_k)$ into $(F^k,\norm{\cdot}_k)$. It is easy to check that $(\mboperators(E, F),\norm{\cdot}_{mb})$ is a normed space, and that it is a Banach space in the case where $F$ is a Banach space; $\norm{\cdot}_{mb}$ is the \emph{multi-bounded norm} on $\mboperators(E,F)$.

Let $E$ and $F$ be multi-normed spaces, and let $T \in \mboperators(E, F)$. It follows immediately from the definitions that $T(B)$ is a multi-bounded set in $F$ whenever $B$ is a multi-bounded set in $E$. Conversely, it is noted in \cite[$\S 6.1.3$]{DP08} that any $T \in \operators(E, F)$ which takes multi-bounded sets to multi-bounded sets is multi-bounded, and, further, that 
\[
\norm{T}_{mb}=\sup\set{ \multibound[T(B)]: B\subset E,\ \multibound(B)\leq 1}\,.
\]
Thus our definitions of $\mboperators(E,F)$ and $\norm{\cdot}_{mb}$ are equivalent to those given in \cite[Definitions 6.9 and 6.12]{DP08}; the definitions in
\cite{DP08}  apply more generally. 

\begin{proposition}\label{multi-bounded operator from lone}
Let $E$ be a multi-normed space, and consider $\lone$ with its minimum multi-norm. Then, for each $T\in\operators(\lone,E)$, we have
\[
	\norm{T}_{mb}=\multibound\set{T(\delta_k):\ k\in\naturals}=\multibound\,T(\lone_{[1]})\,,
\]
so that $\mboperators(\lone,E)=\set{T\in\operators(\lone,E):\ \multibound\set{T(\delta_k):\ k\in\naturals}<\infty}$.
\end{proposition}
\begin{proof}
This follows directly from Lemma \ref{absolutely convex hull and multi-bounded} and the previous paragraph. 
\end{proof}

The following result is immediate; a more general result is given in \cite[Theorem 6.17]{DP08}.

\begin{proposition}\label{operators and minimum mn}
Let $((E^n,\norm{\cdot}_n):\ n\in\naturals)$ and $((F^n,\norm{\cdot}^{\min}_n):\ n\in\naturals)$ be multi-normed spaces. Then each $T\in\operators(E,F)$ is multi-bounded and $\norm{T}_{mb}=\norm{T}$. \enproof
\end{proposition}
 
 For normed spaces $E$ and $F$, the normed space of nuclear operators from $E$ to $F$ is denoted by $(\nuclearoperators(E,F),\nu)$, where $\nu$ is the nuclear norm. It is shown in \cite[Theorem 6.15(ii)]{DP08} that there is a natural contractive inclusion 
 \[
 	(\nuclearoperators(E,F),\nu)\hookrightarrow(\mboperators(E,F),\norm{\cdot}_{mb})\,.
\] 
In particular, $\finiterankoperators(E,F)\subset \mboperators(E,F)$. It is shown in \cite[Example 6.25]{DP08} that the `minimum case', where $\mboperators(E, F) =\nuclearoperators(E, F)$, can occur.

\begin{example}\label{lattice multi-norm}
There are many examples of multi-normed spaces in \cite{DP08}; we shall give some below. An important example is the lattice multi-norm described
 in \cite[$\S4.3$]{DP08}.   Indeed, let $(E,\norm{\cdot})$ be a (complex) Banach
lattice. For $n \in \naturals$, set
\[
	\norm{(x_1,\ldots, x_n)}^L_n =\norm{\abs{x_1}\vee\cdots\vee\abs{x_n}}\qquad (x_1,\ldots,x_n\in E)\,.
\]
Then it is easily checked that $(\norm{\cdot}^L_n:\ n\in\naturals)$ is a multi-norm based on $E$; it is the \emph{lattice multi-norm}. \enproof
\end{example}

\section{Multi-normed spaces as tensor norms}
\label{Multi-normed spaces as tensor norms}

\noindent In this section, we shall show that there are bijections between the families of multi-norm structures
based on a normed space $E$ and certain families of norms on the tensor products $\co\otimes E$ and $\linfty\otimes E$.

Suppose that $E$ and $F$ are normed spaces, and that $\norm{\cdot}$ is a norm on $E\otimes F$. Then $\norm{\cdot}$ is a \emph{sub-cross-norm} if $\norm{x\otimes y}\le \norm{x}\norm{y}$ ($x\in E$, $y\in F$), and a \emph{cross-norm} if 
\[
	\norm{x\otimes y}= \norm{x}\norm{y}\quad (x\in E,\ y\in F)\,.
\]
Further, a sub-cross-norm on $E\otimes F$ is \emph{reasonable} if the  linear functional $\lambda\otimes \mu$ is bounded, with $\norm{\lambda\otimes \mu}\le \norm{\lambda}\norm{\mu}$, for each $\lambda\in\dual{E}$ and $\mu\in\dual{F}$. In fact, each reasonable sub-cross-norm on $E\otimes F$ is a cross-norm, and $\norm{\lambda\otimes \mu}= \norm{\lambda}\norm{\mu}$ for each $\lambda\in\dual{E}$ and $\mu\in\dual{F}$.

The \emph{injective tensor norm} $\|\cdot\|_\varepsilon$ on $E\otimes F$ is defined by
identifying $E\otimes F$ with a subspace of $\operators(\dual{E},F)$; here $x\otimes y$ corresponds to the map $\lambda\mapsto\duality{x}{\lambda}y,\ \dual{E}\to F$. The completion of $E\otimes F$ with respect to this
norm is the \emph{injective tensor product}, denoted  by $E\injectivetensor F$.  
The \emph{projective tensor norm} $\norm{\cdot}_\pi$ on $E\otimes F$ is defined by
\[ 
	\|\tau\|_\pi = \inf \sum_{j=1}^n \|x_j\|\|y_j\|\,,
\]
where the infimum is taken over all representations $\tau=\sum_{j=1}^n x_j\otimes y_j$ of $\tau$ in $E\otimes F$; the completion of $E\otimes F$ with respect to this norm is the \emph{projective tensor product}, denoted by $E\projectivetensor F$. The norms $\norm{\cdot}_\varepsilon$ and $\norm{\cdot}_\pi$ are both reasonable cross-norms on $E\otimes F$, and a norm $\norm{\cdot}$ on $E\otimes F$ is a reasonable cross-norm if and only if 
\[
	\norm{z}_\varepsilon\le\norm{z}\le \norm{z}_\pi\quad(z\in E\otimes F)\,.
\]
For $\mu \in (E\projectivetensor F)'$, define $T_\mu \in \operators(E, \dual{F})$ by  $$
\langle y,\,T_{\mu}x\rangle = \langle x\otimes y,\,\mu\rangle\quad (x\in E,\,y\in F)\,.
$$
Then the map $\mu\mapsto T_\mu,\,\; (E\projectivetensor F)' \to \operators(E, \dual{F}),$ is an isometric isomorphism; we shall identify $(E\projectivetensor F)' $ 
and $\operators(E, \dual{F}).$
See \cite[Chapter II]{DF} and \cite[\S 6.1]{Ryan} for accounts of tensor norms on $E\otimes F$ that include the above remarks.

The following characterization of multi-norms is given in \cite[Theorem 2.35]{DP08}. (There is a similar characterization of dual multi-norms
 in \cite[Theorem 2.36]{DP08}.) 

\begin{theorem}\label{linking different levels of a multi-norm}
Let  $(E, \norm{\cdot})$ be a normed space, and suppose that $ \norm{\cdot}_n$ is a  norm  on $E^n$ for each $ n\in \naturals$,
 with $ \norm{x}_1 =\norm{x}\,\;(x \in E)$.  Then the following are equivalent:\smallskip
 \begin{itemize}
 \item[{\rm (a)}] $(\norm{\cdot}_n : n\in \naturals)$ is a multi-norm  on $E$; \smallskip

\item[{\rm (b)}] $\norm{Tx}_m \leq \norm{T}\norm{x}_n$ for each $T \in \operators(\linfty_n,\linfty_m)$, each $x \in E^n$, and
  each $m,n\in \naturals$. \enproof 
  \end{itemize}
\end{theorem}

We shall now characterize multi-norm spaces in terms of single norms on a certain tensor product; an analogous `coordinate-free' 
characterization of operator spaces is developed in   \cite{Helemskii2010} and \cite{Pi}. 

\begin{definition} 
Let $E$ be a normed space. Then a norm $\norm{\cdot}$ on $\co\otimes E$ is a \emph{$\co$-norm} if $\norm{\delta_1\otimes x}=\norm{x}$ 
for each $x\in E$ and if the linear operator $T\otimes \id_E$ is bounded on $(\co\otimes E,\|\cdot\|)$ with norm at
most $\|T\|$ for each $T\in\compactoperators(\co)$.

Similarly, a norm $\norm{\cdot}$ on $\linfty\otimes E$ is an \emph{$\linfty$-norm} if $\norm{\delta_1\otimes x}=\norm{x}$ for each $x\in E$ 
and if the linear operator $T\otimes \id_E$ is bounded on $(\linfty\otimes E,\|\cdot\|)$ with norm at
most $\|T\|$ for each $T\in\compactoperators(\linfty)$.
\end{definition}

Note that, in the definition of $\co$-norms or $\linfty$-norms above, we use 
$\compactoperators(\co)$ and $\compactoperators(\linfty)$, respectively. 
However, we shall soon see that we can replace $\compactoperators(\co)$ and $\compactoperators(\linfty)$ by the larger spaces $\operators(\co)$
 and $\operators(\linfty)$, respectively.

\begin{lemma}
Let $E$ be a normed space.  
Then each $\co$-norm on $\co\otimes E$ and each $\linfty$-norm on $\linfty\otimes E$ is a reasonable cross-norm.
\end{lemma}
\begin{proof}
Suppose that $\norm{\cdot}$ is a $\co$-norm on $\co\otimes E$. First, given $a,b\in\co$ of the same norm and given $x\in E$, we see that $\norm{b\otimes x}\le\norm{a\otimes x}$ by considering the rank-one operator of norm $1$ in $\compactoperators(\co)$ which maps $a$ to $b$. It follows that
\[
	\norm{a\otimes x}=\norm{a}\norm{\delta_1\otimes x}=\norm{a}\norm{x}\quad(a\in \co,\ x\in E)\,.
\]
From this, it follows from the triangle inequality that $\norm{\tau} \leq \norm{\tau}_\pi$ ($\tau\in \co\otimes E$).

Let $\tau = \sum_{j=1}^k b_j \otimes x_j \in \co\otimes E\subset \ell^{\,\infty}\otimes E$\,; as in equation \eqref{tensor form of finite-rank}, $\tau$ is identified with the map 
\[
	\tau:f\mapsto \sum_{j=1}^k \duality{f}{b_j}x_j \quad\textrm{in}\quad \finiterankoperators(\lone,E)\subset\operators(\lone,E)\,.
\]
Take $f\in\lone$, and consider the rank-one operator $T\in\compactoperators(\co)$ defined by setting  $Tb = \duality{f}{b} \delta_1$ for $b\in \co$, where we regard $\delta_1$ as an element of $\co$. 
Then  we see that $\norm{T}=\norm{f}$ and 
\begin{align*} 
\norm{\tau(f)}&= \Big\| \sum_{j=1}^k \duality{f}{b_j} x_j \Big\| =\Big\|\sum_{j=1}^k \duality{f}{b_j} \delta_1 \otimes x_j \Big\|
=\|(T\otimes \id_E)(\tau)\| \leq \|T\| \|\tau\| =\|f\| \|\tau\|\,,
\end{align*}
using the fact that $\norm{\cdot}$ is a $\co$-norm on $\co\otimes E$. 
We conclude that $\|\tau\|_\varepsilon \leq \|\tau\|$.

Thus $\|\tau\|_\varepsilon \leq \|\tau\|\le \norm{\tau}_\pi$ for each $\tau\in \co\otimes E$, and so $\norm{\cdot}$ is a reasonable cross-norm on $\co\otimes E$. 

The case of an $\linfty$-norm on $\linfty\otimes E$ can be dealt with similarly.
\end{proof}

\begin{theorem}\label{multi-norm as tensor} 
Let $E$ be a normed space. Then there exist bijective correspondences between:
\begin{enumerate}
\item the collection of multi-norms based on $E$;
\item the collection of norms $\norm{\cdot}$ on $\finiterankoperators(\lone,E)$ with the properties that $\norm{\delta_1\otimes x}=\norm{x}$ for each $x\in E$ and that
\begin{align}\label{Eq: multi-norm as tensor 2}
	\norm{S\circ T} \le\norm{T:\lone\to\lone}\norm{S}\quad(T\in\operators(\lone),\ S\in\finiterankoperators(\lone,E))\,;
\end{align}
\item the collection of $\co$-norms on $\co\otimes E$; and
\item the collection of $\linfty$-norms on $\linfty\otimes E$.
\end{enumerate}
\end{theorem}
\begin{proof}
Let $(\norm{\cdot}_n:\ n\in\naturals)$ be a multi-norm based on $E$. Consider the minimum multi-norm on $\lone$, so that, by Proposition \ref{operators and minimum mn},  $\mboperators(\lone)=\operators(\lone)$ and $\norm{T}_{mb}=\norm{T}$ for $T\in\operators(\lone)$. Since $\finiterankoperators(\lone,E)\subset\mboperators(\lone,E)$, we can consider the norm $\norm{\cdot}$ on $\finiterankoperators(\lone,E)$ to be the restriction of the norm $\norm{\cdot}_{mb}$ of $\mboperators(\lone,E)$. For each $x\in E$, we have, by Proposition \ref{multi-bounded operator from lone},
\[
	\norm{\delta_1\otimes x}=\norm{\delta_1\otimes x}_{mb}=\multibound\set{(\delta_1\otimes x)(\delta_k):\ k\in\naturals}=\norm{x}\,.
\]
 It now follows easily that the norm $\norm{\cdot}$ on $\finiterankoperators(\lone,E)$ satisfies \eqref{Eq: multi-norm as tensor 2}, and hence $\norm{\cdot}$ satisfies the conditions in clause (ii).

Since $\co\otimes E\subset\finiterankoperators(\lone,E)$, once we have a norm $\norm{\cdot}$ on $\finiterankoperators(\lone,E)$ with the properties as stated in clause (ii), we can give $\co\otimes E$ the norm which is the restriction of the norm $\norm{\cdot}$ on $\finiterankoperators(\lone,E)$. It is easily checked that this norm is a $\co$-norm on $\co\otimes E$; in fact, for each $T\in\operators(\co)$, the linear operator $T\otimes \id_E$ is bounded on $(\co\otimes E,\|\cdot\|)$ with norm at most $\norm{T}$.

Using the identification of $\finiterankoperators(\lone,E)$ with $\linfty\otimes E$ given in equation \eqref{tensor form of finite-rank}, we can give $\linfty\otimes E$ a norm $\norm{\cdot}$ with properties similar to those of the $\linfty$-norms, except that $\compactoperators(\linfty)$ is replaced by the set $\set{\dual{S}:\ S\in\operators(\lone)}$. Let $T\in\operators(\linfty)$ and $\sigma\in\linfty\otimes E$. We wish to show that
\[
	\norm{(T\otimes \id_E)(\sigma)}\le\norm{T}\norm{\sigma}\,.
\]
Indeed, suppose that $\sigma=\sum_{i=1}^n a_i\otimes x_i$, where $a_1,\ldots,a_n\in\linfty$ and $x_1,\ldots,x_n\in E$, and take $\varepsilon>0$.  Let $F$ be the linear span of the set $\set{a_1,\ldots, a_n}$ in $\linfty$, set $G=T(F)$, and consider  \mbox{$T_0:F\to G$} to be the restriction of $T$. We can identify $\dual{F}$ and $\dual{G}$ as $\lone/X$ and $\lone/Y$, respectively, where 
\[
	X=\set{f\in\lone:\ f\,|\,F=0}\quad \textrm{and}\quad Y=\set{f\in\lone:\ f\,|\,G=0}
\] 
are closed subspaces of finite codimension in $\lone$. By the projectivity of $\lone$ \cite[p. 72]{DF}, there is an operator $S\in\operators(\lone)$ such that the following diagram commutes:
\begin{equation*}
\SelectTips{eu}{12}\xymatrix{\lone \ar[r]^-{S} \ar@{->>}[d]_{\pi_Y} & \lone \ar@{->>}[d]^{\pi_X} \\ 
            \lone/Y=\dual{G} \ar[r]^-{\dual{T_0}}& \lone/X=\dual{F}\,; }
\end{equation*}
moreover, we can choose $S$ such that $\norm{S}\le\norm{\dual{T_0}}+\varepsilon$. It follows that 
 $\dual{S}$ and $T$ agree on $F$ and that $\norm{S}\le \norm{T}+\varepsilon$. Thus
 \[
 	\norm{(T\otimes \id_E)(\sigma)}=\norm{(\dual{S}\otimes \id_E)(\sigma)}\le\norm{S}\norm{\sigma}\le(\norm{T}+\varepsilon)\norm{\sigma}\,.
 \]
Letting $\varepsilon\searrow 0$, we obtain the desired inequality. 
In particular, this shows that $\norm{\cdot}$ is an $\linfty$-norm on $\linfty\otimes E$.

Thus, we have constructed maps from the collection of multi-norms based on $E$ into the collections specified in the clauses (i), (ii), and (iii). 

To define the (proposed) inverses of these maps, first suppose that we already have  a $\co$-norm  $\norm{\cdot}$ on $\co\otimes E$. Then we define
\[ 
	\|(x_1,x_2,\ldots,x_n)\|_n = \Big\| \sum_{j=1}^n \delta_j \otimes x_j \Big\| \qquad \big(x_1,\ldots,x_n \in E,\ n\in\naturals \big)\,. 
\]
Clearly $\norm{\cdot}_n$ is a norm on $E^n$ for each $n\in\naturals$, and it is easy to see that clause (b) of Theorem \ref{linking different levels of a multi-norm} is satisfied, and so it follows from Theorem \ref{linking different levels of a multi-norm} that $(\norm{\cdot}_n : n\in \naturals)$ is a multi-norm  on $E$.

To prove that the correspondences defined above are bijections, it is sufficient to show that any given  multi-norm
 $(\norm{\cdot}_n:\ n\in\naturals)$ based on $E$ determines uniquely an $\linfty$-norm  $\norm{\cdot}$ on $\linfty\otimes E$ and a $\co$-norm  $\norm{\cdot}$  on $\co\otimes E$ such that
\begin{align}\label{Eq: multi-norm as tensor 1} 
	\Big\| \sum_{j=1}^n \delta_j \otimes x_j \Big\|=\|(x_1,x_2,\ldots,x_n)\|_n  \qquad \big(x_1,\ldots,x_n \in E ,\ n\in\naturals \big)\,. 
\end{align}
So, let $\sigma=\sum_{i=1}^n a_i\otimes x_i$ be an element of $\linfty\otimes E$, and take $\varepsilon>0$. Let $F$ be the linear span 
of the set $\set{a_1,\ldots,a_n}$ in $\linfty$. Then there exist $N\in\naturals$ and a subspace $G$ of $\linfty_N$ such that the Banach--Mazur 
distance \mbox{$d(F,G)<1+\varepsilon$,} and so there exists $T_0:F\to G$ with $\norm{T_0}<1+\varepsilon$ and $\norm{T_0^{-1}}=1$.
 The injectivity of $\linfty_N$ and $\linfty$ then implies that $T_0$ and $T_0^{-1}$ extend to linear operators $T$ and $S$, respectively, 
in $\operators(\linfty)$ with $\norm{T}=\norm{T_0}<1+\varepsilon$ and $\norm{S}=\norm{T_0^{-1}}=1$ and with the range of $T$ contained in
 $\linfty_N$.  For each $i\in\naturals_n$, set $a_{i,\varepsilon}=T_0a_i$, and then set 
\[
	\sigma_\varepsilon=\sum_{i=1}^n a_{i,\varepsilon}\otimes x_i \in\linfty_N\otimes E\,.
\]
  It follows that
\[
	\sigma_\varepsilon=(T\otimes\id_E)(\sigma) \quad\textrm{and}\quad \sigma=(S\otimes \id_E)(\sigma_\varepsilon)=(STS\otimes\id_E)(\sigma_\varepsilon)\,.
\]
Note that both $T$ and $STS$ belong to $\compactoperators(\linfty)$, and so the $\linfty$-norm property implies that 
\[
	(1+\varepsilon)^{-1}\norm{\sigma}\le\norm{\sigma_\varepsilon}\le (1+\varepsilon)\norm{\sigma}\,. 
\]
Thus  we obtain $\norm{\sigma}=\lim_{\varepsilon\searrow 0}\norm{\sigma_\varepsilon}$. Further note that, since $\sigma_\varepsilon\in\coo\otimes E$, 
by \eqref{Eq: multi-norm as tensor 1}, its norm is totally determined by the multi-norm $(\norm{\cdot}_n:\ n\in\naturals)$. Hence, the $\linfty$-norm
 $\norm{\cdot}$ on $\linfty\otimes E$ is determined completely by the given multi-norm on $E$. 

For the case of a $\co$-norm $\norm{\cdot}$ on $\co\otimes E$, notice that, for each $\sigma\in\co\otimes E$, we have
\[
	\norm{\sigma}=\lim_{n\to\infty}\norm{(P_n\otimes \id_E)(\sigma)},
\]
where $P_n\in\compactoperators(\co)$ is the projection onto the first $n$ coordinates. We see, again by \eqref{Eq: multi-norm as tensor 1}, 
 that the norm of $(P_n\otimes \id_E)(\sigma)$  is  determined by the multi-norm $(\norm{\cdot}_n:\ n\in\naturals)$.
\end{proof}

Thus, in particular, the study of multi-norms based on a normed space $E$ is equivalent to the study of $\co$-norms on $\co\otimes E$. 

\begin{remark}
Let $E$ be a multi-normed space with the associated $\co$-norm $\norm{\cdot}$ on $\co\otimes E$. From the theorem above, it follows that,
 for each $\sigma\in\co\otimes E$, we have
\[
	\norm{\sigma}=\norm{\sigma}_{\mboperators(\lone,E)}=\multibound\set{\sigma(\delta_k):\ k\in\naturals}\,,
\]
where, in the last two terms, $\sigma$ is considered as an element of $\finiterankoperators(\lone,E)$.
\end{remark}

Theorem \ref{multi-norm as tensor} and its proof imply the following.

\begin{corollary}
Let $E$ be a normed space. 
\begin{enumerate}
\item Suppose that $\norm{\cdot}$ is a $\co$-norm on $\co\otimes E$. Then, for each $T\in\operators(\co)$,
the operator $T\otimes\id_E$ is bounded on $(\co\otimes E, \|\cdot\|)$, with norm $\|T\|$. 
\item Suppose that $\norm{\cdot}$ is an $\linfty$-norm on $\linfty\otimes E$. Then, for each $T\in\operators(\linfty)$,
the operator $T\otimes\id_E$ is bounded on $(\linfty\otimes E, \|\cdot\|)$, with norm $\|T\|$.\enproof
\end{enumerate}
\end{corollary}

\begin{corollary}  Let $E$ be a normed space.
The maximum multi-norm structure based on $E$ corresponds to the projective tensor norm  on $\co\otimes E$,
and the minimum multi-norm structure based on $E$  corresponds to the injective tensor norm on $\co\otimes E$. \enproof
\end{corollary}

A related result about the maximum multi-norm is proved by more elementary calculations in \cite[Theorem 3.43]{DP08}.

\begin{remark}
As discussed in \cite[$\S 2.4.5$]{DP08}, it follows  from the previous corollary that our notion of a 
$\co$-norm is equivalent to  that of a norm on $\co\otimes E$ satisfying `condition (P)' of Pisier. This condition is the main topic of the memoir \cite{MN}. 
The following theorem gives a general \emph{representation theorem for multi-normed spaces}. It shows a universal property of the lattice
 multi-norms described in Example \ref{lattice multi-norm}; the result follows from a theorem of Pisier stated as \cite[Th\'{e}or\`{e}me 2.1]{MN}. 
We are indebted to the late Professor Nigel Kalton for the reference to \cite{MN}.
\end{remark}

\begin{theorem} Let $((E_n,\norm{\cdot}_n):\ n \in \naturals)$ be a
multi-Banach space. Then there is a Banach lattice $X$ and an isometric
embedding $J: E \to X$ such that, for each $n \in \naturals$, we have
\[
	\norm{(Jx_1,\ldots,Jx_n)}^L_n = \norm{(x_1,\ldots, x_n)}_n\qquad (x_1,\ldots,x_n \in E)\,. \tag*{\enproof}
\]
\end{theorem}

\begin{proposition}\label{multi-bounded and weakly compact}
Let $E$ be a multi-normed space, and assume that  $\mboperators(\ell^1,E)$ consists of weakly compact operators.
Then every multi-bounded subset of $E$ is relatively weakly compact.
\end{proposition}
\begin{proof}
By the Eberlein-\u{S}mulian theorem \cite[2.4.6]{Megginson}, we need to consider only  a countable multi-bounded subset 
$\set{x_n\colon n\in\naturals}$ of $E$. In this case, it then follows that the map $T:\delta_n\mapsto x_n$ extends to a bounded linear operator
 $T:\lone\to E$ which in fact belongs to $\mboperators(\lone,E)$. Thus, by the assumption, $T$ must be weakly compact. In particular,
 $\set{x_n\colon n\in\naturals}\subset T(\lone_{[1]})$ is relatively weakly compact.
\end{proof}

\section{Dual multi-normed spaces}
\label{Dual multi-normed spaces}

\noindent In this section, we shall quickly sketch how \emph{dual multi-normed spaces}
(see \cite[$\S 2.1.2$]{DP08})  fit into a tensor--norm framework.  

\begin{definition}
Let $E$ be a normed space. Then a norm $\norm{\cdot}$ on $\lone\otimes E$ is an \emph{$\lone$-norm} if $\norm{\delta_1\otimes x}=\norm{x}$ for each 
$x\in E$ and if the linear operator $T\otimes \id_E$ is bounded on $(\lone\otimes E,\|\cdot\|)$ with norm at
most $\|T\|$ for each $T\in\compactoperators(\lone)$.
\end{definition}

Note that, again, an $\lone$-norm on $\lone\otimes E$ is necessarily a reasonable cross-norm.

In fact, there is an analogue of Theorem \ref{multi-norm as tensor} that relates dual multi-norms based on a normed space $E$ to $\lone$-norms on 
$\lone\otimes E$. We shall not use this result, and so omit the details, but we note that the necessary preliminary results are contained in
 \cite[$\S2.3.2$]{DP08}, where the basic results on the relations between multi-norms and dual multi-norms are obtained in a different way.

Let $E$ be a multi-normed space, and consider its associated $\linfty$-norm $\|\cdot\|$ on $\linfty\otimes E$.  Since $\|\cdot\|\leq\|\cdot\|_\pi$,
there is a dense-range contraction $$(\linfty\otimes E,\norm{\cdot}_\pi)\rightarrow  (\linfty\otimes E,\norm{\cdot})\,.$$  The
dual of this map is an injective contraction 
\[
	\dual{(\linfty\otimes E)} \hookrightarrow \dual{(\linfty\projectivetensor E)}=\operators(\linfty,\dual{E})\,,
\] 
and so we can identify $\dual{(\linfty\otimes E)}$ with a
subspace of $\operators(\linfty,\dual{E})$; this subspace is denoted by $\operators_{\beta}(\linfty,\dual{E})$, where $\beta$
is the norm induced by $\dual{(\linfty\otimes E)}$.

The space $\lone\otimes \dual{E}$ acts linearly on $\linfty\otimes E$ by the specification
\[ 
	\duality{b\otimes x}{\,a\otimes \lambda} = \duality{b}{a} \duality{x}{\lambda}
\qquad (a\otimes\lambda\in\lone\otimes \dual{E},\ b\otimes x\in \linfty\otimes E)\,. 
\]
Since the $\linfty$-norm on $\linfty\otimes E$ is reasonable, we see that the action of $\lone\otimes \dual{E}$ on $\linfty\otimes E$ is continuous. 
Thus  we have a natural inclusion  $\lone\otimes\dual{E}\hookrightarrow \dual{(\linfty\otimes E)}$.

Now consider the $\co$-norm $\norm{\cdot}$ on $\co\otimes E$ that is associated with the given multi-norm based on $E$. Then the natural 
 inclusion $(\co\otimes E,\norm{\cdot})\hookrightarrow(\linfty\otimes E,\norm{\cdot})$ is isometric, and therefore we have a natural quotient 
mapping $\dual{(\linfty\otimes E)}\twoheadrightarrow\dual{(\co\otimes E)}$.
We note that $\lone\otimes E$ acts on $\co\otimes E$ by restricting its action on $\linfty\otimes E$ to $\co\otimes E$, that this new action also 
induces an inclusion $\lone\otimes\dual{E}\hookrightarrow \dual{(\co\otimes E)}$, and that the following diagram commutes:
\begin{align}\label{Eq: ell1_prop}
\SelectTips{eu}{12}\xymatrix{\lone\otimes \dual{E}\ \ar@{^(->}[r] \ar@{^(->}[dr] & \dual{(\linfty\otimes E)} \ar@{->>}[d] \\ 
            &\ \dual{(\co\otimes E)}\,.}
\end{align}

\begin{theorem}\label{ell1_prop}
Let $E$ be a multi-normed space.  Then 
there exists a unique $\lone$-norm on $\lone\otimes \dual{E}$ such that the two inclusions in the diagram \eqref{Eq: ell1_prop} are isometric. 
 Furthermore, the image of $\lone\otimes \dual{E}$ in $\dual{(\linfty\otimes E)}$ is a norming set for $\linfty\otimes E$.
\end{theorem}
\begin{proof}
The uniqueness statement is obvious, and the only way to define this $\lone$-norm on $\lone\otimes\dual{E}$ is by taking the restriction of the 
norm of $\dual{(\linfty\otimes E)}$. The fact that the norm on $\lone\otimes \dual{E}$ just defined is an $\lone$-norm follows directly from the 
properties of the $\linfty$-norm on $\linfty\otimes E$.

We shall now prove that the inclusion $\lone\otimes\dual{E}\hookrightarrow \dual{(\co\otimes E)}$ is also isometric. Thus consider 
$\sigma=\sum_{i=1}^m a_i\otimes \lambda_i\in\lone\otimes \dual{E}$, and take $\varepsilon>0$. By the definition of the norm on $\lone\otimes\dual{E}$, 
we see that there exists $\tau=\sum_{j=1}^n b_j\otimes x_j\in\linfty\otimes E$ with $\norm{\tau}=1$ and such that
 $\duality{\tau}{\,\sigma}>\norm{\sigma}-\varepsilon$. For each $k\in\naturals$, let $P_k\in\operators(\linfty)$ be the projection onto the first 
$k$ coordinates. Then it is clear that 
\[
	\norm{\sigma}-\varepsilon<\duality{\tau}{\,\sigma} = 
\sum_{i=1}^m \sum_{j=1}^n \duality{a_i}{b_j}\duality{x_j}{\lambda_i}=\lim_{k\to\infty}\duality{(P_k\otimes\id_E)(\tau)}{\,\sigma}\,;
\]
further, we note that $(P_k\otimes\id_E)(\tau)\in\co\otimes E$ with $\norm{(P_k\otimes\id_E)(\tau)}\le \norm{\tau}=1$. This proves that the  inclusion map $\lone\otimes\dual{E}\hookrightarrow \dual{(\co\otimes E)}$ is isometric.

For the last statement, take $\tau=\sum_{i=1}^n a_i\otimes x_i\in \linfty\otimes E$. By the Hahn--Banach theorem, there exists  $T\in\operators_{\beta}(\linfty,\dual{E})$
with $\beta(T)=1$ and $\duality{\tau}{T} = \norm{\tau}$.  Let $F$ be the linear span of the set $\set{a_1,\ldots,a_n}$ in $\linfty$, and take $\varepsilon>0$. Arguing as in the proof of Theorem \ref{multi-norm as tensor}, we see that there exist bounded linear operators $R,S:\linfty\to\linfty$ such that the range of $R$ is contained in $\linfty_N$ for some natural number $N$, such that $\norm{S}\norm{R}<1+\varepsilon$, and such that $(SR)|F=\id_F$; moreover, we can arrange that $R=\dual{U}$ for some $U\in\operators(\lone)$. Then $TSP_N$ belongs to $\lone\otimes \dual{E}$ (considered canonically as a subspace of $\finiterankoperators(\linfty,\dual{E})\subset\operators_\beta(\linfty,\dual{E})$), and so
\[
	TSR=TSP_N\dual{U}\in \lone\otimes\dual{E}\,.
\]
We see that $\beta(TSR)\le\beta(T)\norm{SR}<1+\varepsilon$ and that
\[
	\duality{\tau}{TSR}=\duality{\tau}{T}=\norm{\tau}\,.
\]
This holds true for each $\varepsilon>0$, and so $\lone\otimes\dual{E}$ is a norming set for $\linfty\otimes E$, as claimed.
\end{proof}

Analogously, when $E$ is a dual multi-normed space, we see that 
$\linfty\otimes \dual{E}$ acts on $\lone\otimes E$ by an action which satisfies the condition that
\[ 
 \duality{b\otimes x}{\,a\otimes \lambda} = \duality{a}{b} \duality{\lambda}{x}
\quad (a\otimes\lambda\in \linfty\otimes \dual{E},\ b\otimes x\in\lone\otimes E)\,. 
\]
The proof of the following theorem is essentially the same as that of Theorem \ref{ell1_prop}.

\begin{theorem}
Let $E$ be a dual multi-normed space.  Then the above dual pairing gives an injection
$\linfty\otimes \dual{E}\rightarrow \dual{(\lone\otimes E)}$ which induces a multi-norm
structure on $\dual{E}$.  Furthermore, $\co\otimes \dual{E}$ is a norming set for $\lone\otimes E$. \enproof
\end{theorem}

Suppose that $E$ is a multi-normed space. Then the above two theorems give us an
$\lone$-norm on $\lone\otimes \dual{E}$ and then a $\co$-norm on $\co\otimes \bidual{E}$.    
Thus we have the following (see also \cite[Theorem 2.34]{DP08}, where the proof  is given by a different argument).

\begin{corollary}
Let $((E^n,\norm{\cdot}_n):\ n\in\naturals)$ be a multi-normed space, and induce a dual multi-norm structure based on $\dual{E}$.
Using this, induce a multi-norm $(\norm{\cdot}_n'':\ n\in\naturals)$ based  on $\bidual{E}$.  Then, for each $n\in\naturals$, 
the restriction of $\norm{\cdot}_n''$ to the canonical image of $E^n$ in $(\bidual{E})^n$ is equal to $\norm{\cdot}_n$.  
\enproof
\end{corollary}

\section{The $(p,q)$-multi-norm}

\noindent Following \cite[$\S 4.1$]{DP08}, we now introduce an important class of multi-norms.  Let $E$ be a normed space, and take $p,q$ with $1\leq p,q< \infty$. For each $n \in \naturals$ and each $\tuple{x}=(x_1,\ldots, x_n) \in E^n$, we define
\[
\norm{\tuple{x}}^{(p,q)}_n=\sup \set{ \lp\sum_{i=1}^n\abs{\duality{x_i}{\lambda_i}}^{q}\rp^{1/q}: \tuple{\lambda}=(\lambda_1,\ldots,\lambda_n) \in (\dual{E})^n,\, \mu_{p,n}(\tuple{\lambda})\leq 1}\,.
\]
It is clear that $\norm{\cdot}^{(p,q)}_n$ is a norm on $E^n$. As proved in \cite[Theorem 4.1]{DP08}, in the case where $1\le p\le q<\infty$, the sequence $(\norm{\cdot}^{(p,q)}_n: n \in \naturals)$ is a multi-norm based on $E$.

\begin{definition}
Let $E$ be a normed space, and take $p,q$ with $1\leq p\leq q< \infty$. Then the multi-norm $(\norm{\cdot}^{(p,q)}_n: n \in \naturals)$ described above is the \emph{$(p,q)$-multi-norm} over $E$.

A subset of $E$ is \emph{$(p,q)$-multi-bounded} if it is multi-bounded with respect to the $(p,q)$-multi-norm. The \emph{$(p,q)$-multi-bound} of such a set $B$ is denoted by $\multibound_{p,q}(B)$.
\end{definition}

\begin{remark}\label{(1,1) multi-norm is maximum}
By \cite[Theorem 4.6]{DP08} ({\it cf.} \eqref{max-mn}), the $(1,1)$-multi-norm is just the maximum multi-norm based on $E$.
\end{remark}

\begin{lemma}\label{weak (p,q)-multi-norm on dual space}
Let $E$ be a normed space, and take $p,q$ with $1\leq p,q< \infty$. Then, for each $n \in \naturals$ and $\tuple{\lambda}=(\lambda_1,\ldots,\lambda_n)\in (\dual{E})^n$, we have
\begin{align*}
\norm{\tuple{\lambda}}^{(p,q)}_n
=\sup \set{ \lp\sum_{i=1}^n\abs{\duality{x_i}{\lambda_i}}^{q} \rp^{1/q}: \tuple{x}=(x_1,\ldots,x_n) \in E^n,\, \mu_{p,n}(\tuple{x})\leq 1}\,.
\end{align*}
\end{lemma}
\begin{proof}
This is proved in \cite[Proposition 4.10]{DP08}; it follows from the Principal of Local Reflexivity.
\end{proof}

This lemma implies that, for each normed space $E$, the $(p,q)$-multi-norm based  on $E$ is the same as the one induced from the $(p,q)$-multi-norm based   on $\bidual{E}$. 
 
Suppose now that $E$ is a normed space and that $p,q$ satisfy $1\le p\le q<\infty$. 

\begin{definition}
We denote by $\operators_{p,q}(\lone, E)$ the subset of $\operators(\lone,E)$ consisting of those operators $T$ with the property that $\set{T(\delta_k)\colon k\in\naturals}$ is $(p,q)$-multi-bounded in $E$. We define the norm on $\operators_{p,q}(\lone, E)$ by
\[
	\alpha_{p,q}(T):=\multibound_{p,q}\set{T(\delta_k)\colon k\in\naturals}\,.
\]
\end{definition}

By Proposition \ref{multi-bounded operator from lone}, we see that $\operators_{p,q}(\lone, E)=\mboperators(\lone,E)$ 
when $\lone$ is given the minimum multi-norm and $E$ is given the $(p,q)$-multi-norm; moreover,
\[
	\alpha_{p,q}(T)=\norm{T}_{mb}\quad (T\in \operators_{p,q}(\lone,E))\,.
\]
In particular, it follows that 
\[
	\finiterankoperators(\lone,E)\subset \operators_{p,q}(\lone, E)\subset \operators(\lone, E)\,,
\] 
and that, indeed, $(\operators_{p,q}(\lone, E), \alpha_{p,q})$ is a normed space; it is a Banach space when $E$ is a Banach space. 

From the discussion above, we see also that the natural injection from 
 $\linfty\otimes E$ into $(\operators_{p,q}(\lone, E), \alpha_{p,q})$ is  isometric with respect to the  $\linfty$-norm on $\linfty\otimes E$ associated with the $(p,q)$-multi-norm.
 
Recall from \cite[Chapter 10]{DJT} 
that an operator $T$ from a normed space $E$ into another normed space $F$ is  \emph{$(q,p)$-summing} if there exists a constant $C$ such that
\[
	\left(\sum_{i=1}^n\norm{Tx_i}^q\right)^{1/q}\le C\,\mu_{p,n}(x_1,\ldots, x_n)\quad (x_1,\ldots, x_n\in E,\ n\in\naturals)\,.
\]
The smallest such constant $C$ is denoted by $\pi_{q,p}(T)$. The set of $(q,p)$-summing operators, denoted by $\Pi_{q,p}(E,F)$,  is a normed space when equipped with the norm $\pi_{q,p}$\,; it is a Banach space when $E$ and $F$ are Banach spaces. When $p=q$, we shall write $\operators_p$, $\Pi_p$, and $\pi_p$ instead of $\operators_{p,p}$, $\Pi_{p,p}$, and $\pi_{p,p}$, respectively. The space $(\Pi_p,\pi_p)$ consisting of all $p$-summing operators has been studied by many authors; see \cite{DF}, \cite{DJT}, and \cite{Jameson}, for example.

\begin{proposition}\label{connection with (q,p)-summing operators} 
Let $E$ be normed space, and take $p,q$ with $1\le p\le q<\infty$. Suppose that $T\in\operators(\lone,E)$. Then $T\in\operators_{p,q}(\lone,E)$ if and only if $\dual{T}\in\Pi_{q,p}(\dual{E},\linfty)$. In this case, we have $\alpha_{p,q}(T)=\pi_{q,p}(\dual{T})$.
\end{proposition}
\begin{proof}
Suppose that $T\in \operators_{p,q}(\lone,E)$. From the previous discussion,  $\alpha_{p,q}(T)=\multibound_{p,q}T(\lone_{[1]})$, and so $\alpha_{p,q}(T)$ is the smallest constant $C$ such that
\begin{align*}
C\,\mu_{p,n}(\lambda_1,\ldots,\lambda_n)\ge\left(\sum_{i=1}^n\abs{\duality{Ta_i}{\lambda_i}}^q\right)^{1/q}=\left(\sum_{i=1}^n\abs{\duality{a_i}{\dual{T}\lambda_i}}^q\right)^{1/q}\,
\end{align*}
for every $n\in\naturals$, $a_1,\ldots, a_n\in \lone_{[1]}$, and $\lambda_1$,\ldots, $\lambda_n\in\dual{E}$. 
Taking the supremum over all elements $a_1,\ldots, a_n\in \lone_{[1]}$, we see that $\alpha_{p,q}(T)$ is the smallest constant $C$ such that
\begin{align*}
C\,\mu_{p,n}(\lambda_1,\ldots,\lambda_n)\ge\left(\sum_{i=1}^n\norm{\dual{T}\lambda_i}^q\right)^{1/q}. 
\end{align*}
Thus $\dual{T}\in\Pi_{q,p}(\dual{E},\linfty)$ and $\pi_{q,p}(\dual{T})=\alpha_{p,q}(T)$. 

The converse follows in the same way.
\end{proof}

In the following result, we shall use the fact that $\Lone_\reals(\Omega)$ is an AL-space as a (real) Banach lattice, and so its dual space is an AM-space with an order-unit. Thus there is a compact space $K$ and a linear isometry $\theta:\dual{\Lone(\Omega)}\to\C(K)$ such that $\theta |{\dual{\Lone_\reals(\Omega)}}$ is an order-isomorphism from $\dual{\Lone_\reals(\Omega)}$ onto $\C_\reals(K)$: this is Kakutani's representation theorem. 
For these  results on Banach lattices, see \cite[\S 12]{AB}, for example.

\begin{theorem}\label{equivalence of (p,q)-multi-norm on Lone}
Let $(\Omega,\mu)$ be a measure space, and take $p,q,r$ with $1\le p<q<r<\infty$. Then the $(p,q)$-multi-norm is equivalent to the $(1,q)$-multi-norm 
based on $\Lone(\Omega)$ and dominates 
the $(r,r)$-multi-norm  based on $\Lone(\Omega)$. 
\end{theorem}
\begin{proof}
By \cite[Theorem 10.9]{DJT}, we have
\[
\Pi_{q,p}(\C(K),\linfty)=\Pi_{q,1}(\C(K),\linfty)\subset\Pi_{r}(\C(K),\linfty)
\]
for each compact space $K$, where the last inclusion is continuous. The conclusion then follows from Proposition \ref{connection with (q,p)-summing operators}.
\end{proof}

We shall consider the mutual equivalence of various $(p,q)$-multi-norms
and some other multi-norms based on certain Banach spaces in \cite{DDPR2}.

\begin{theorem}\label{(p,p)-multi-bounded and weak compactness} 
Let $E$ be a Banach space and take $p\in [1,\infty)$. Then every $(p,p)$-multi-bounded subset of $E$ is relatively weakly compact.
\end{theorem}
\begin{proof}
Let $T\in\mboperators(\lone,E)=\operators_p(\lone,E)$. By Proposition \ref{connection with (q,p)-summing operators}, $\dual{T}\in \Pi_p(\dual{E},\linfty)$. 
By the Pietsch Factorization Theorem, every $p$-summing operator is weakly compact \cite[Theorem 2.17]{DJT}. It follows that $\dual{T}$ is weakly compact, and hence so is $T$. By Proposition \ref{multi-bounded and weakly compact}, every $(p,p)$-multi-bounded subset of $E$ must be relatively weakly compact.
\end{proof}

\begin{corollary}\label{multi-bounded and weak compactness in Lone}
Let $(\Omega,\mu)$ be a measure space, and take $p,q$ with $1\le p\le q<\infty$. Then every $(p,q)$-multi-bounded subset of $\Lone(\Omega)$ is relatively weakly compact. 
\end{corollary}
\begin{proof}
This is a consequence of Theorems \ref{equivalence of (p,q)-multi-norm on Lone} and \ref{(p,p)-multi-bounded and weak compactness}.
\end{proof}

\begin{remark} 
Theorem \ref{(p,p)-multi-bounded and weak compactness} cannot be generalized to $(p,q)$-multi-bounded sets in the case where $p<q$. Indeed, it is a result of Kwapie\'{n} and Pe{\l}czy\'{n}ski that 
\[
	S:(\alpha_n)\mapsto \left(\sum_{i=1}^n\alpha_i\right)_{n=1}^\infty,\quad \lone\to\linfty\,,
\] 
is $(q,p)$-summing for every $1\le p<q<\infty$, but $S$ is not weakly compact ({\it cf.} \cite[p. 210]{DJT}). Consider the operator 
$T\in\operators(\lone,\co)$ defined by requiring that 
$$
T(\delta_n)= \sum_{i=1}^n\delta_i\quad(n\in\naturals)\,.
$$
Then $\dual{T}=S$. In particular,  $T$ is not weakly compact, and so, by the Kre\u{\i}n-\u{S}mulian theorem \cite[Theorem 2.8.14]{Megginson}, 
the set $\set{T(\delta_n):\ n\in\naturals}$ is not relatively weakly compact. However, it follows from Proposition \ref{connection with
  (q,p)-summing operators} that $T\in\operators_{p,q}(\lone,\co)$, and so $\set{T(\delta_n):\ n\in\naturals}$ is $(p,q)$-multi-bounded. 
Thus we obtain a subset of $\co$ which is $(p,q)$-multi-bounded for every $1\le p<q<\infty$, but  which is not relatively weakly compact.
\end{remark}

\section{The standard $q$-multi-norm on $\Lspace^p$ spaces}
\label{section: standard (p,q)-multi-norm}

\noindent Let $(\Omega,\mu)$ be a measure space, and take $p,q$ with $1\leq p\leq q<\infty$. In \cite[$\S4.2$]{DP08}, there is a definition and discussion of the \emph{standard $q$-multi-norm} on $E:=\Lspace^p(\Omega)$. We recall the definition. 

Take $n\in\naturals$.  For each partition $\tuple{X}=\set{X_1,\dots,X_n}$ of $\Omega$ into measurable subsets and each 
$f_1,\ldots,f_n\in \Lspace^p(\Omega)$, we define
\[ 
		\norm{(f_1,\ldots,f_n)}^{[q]}_n =\sup_{\tuple{X}} \Big( \sum_{i=1}^n \|P_{X_i}f_i\|^q \Big)^{1/q}. 
\]
Here $P_{X_i}:f\mapsto f\chi_{X_i}$ is the projection of $\Lspace^p(\Omega)$ onto
$\Lspace^p(X_i)$, $\norm{\cdot}$ is the $\Lspace^p$-norm, 
and the supremum is taken over all such measurable partitions $\tuple{X}$ of $\Omega$. It is verified in \cite[$\S 4.2.1$]{DP08} 
that $(\norm{\cdot}^{[q]}_n:\ n\in\naturals)$ is a multi-norm based on
$\Lspace^p(\Omega)$; it is called the standard $q$-multi-norm on $\Lspace^p(\Omega)$ in \cite[Definition 4.21]{DP08}.

In terms of tensor
norms, we have the following, which applies in the special case where $q=p$.

\begin{theorem}
Let $\Omega$ be a measure space, and take $p\ge 1$. Then the standard $p$-multi-norm induces the $\co$-norm on $\co\otimes \Lspace^p(\Omega)$
which comes from identifying $\co\otimes \Lspace^p(\Omega)$ with a subspace of
the vector-valued space $\Lspace^p(\Omega,\co)$.
\end{theorem}
\begin{proof}
We set $F=\Lspace^p(\Omega,\co)$. Take $n\in\naturals$, and let $f_1,\ldots,f_n\in \Lspace^p(\Omega)$, so that
\[ 
\Big\| \sum_{i=1}^n \delta_i\otimes f_i \Big\|_{F}^p =
\int_\Omega \Big\| \sum_{i=1}^n\delta_i\otimes f_i(t)\Big\|_{\co}^p \ \dd m(t)
= \int_\Omega \max_{i\in\naturals_n} |f_i(t)|^p \ \dd m(t)\,. 
\]
For each $i\in\naturals_n$, let $Y_i$ be the set of points of $\Omega$ at which 
$|f_i|$ equals $\max\set{\abs{f_j}:\ j\in\naturals_n}$. Set $X_1=Y_1$ and $X_j=Y_j\setminus\bigcup_{i=1}^{j-1} Y_i$ for each $j\in\set{2,\ldots,n}$, so that $\set{X_1,\ldots,X_n}$ is a measurable partition of $\Omega$. Then we see that
\[ 
	\sum_{i=1}^n \|\chi_{X_i}f_i\|^p = \sum_{i=1}^n \int_{X_i} |f_i(t)|^p\ \dd m(t)
	= \int_\Omega \max_{i\in\naturals_n} |f_i(t)|^p \ \dd m(t). \]
Thus 
$\norm{(f_1,\ldots,f_n)}^{[p]}_n\ge \norm{\sum_{i=1}^n \delta_i\otimes f_i}_F$.

On the other hand, for each measurable partition $\tuple{X}=\set{X_1,\ldots,X_n}$ of $\Omega$, we have
\[ 
\sum_{i=1}^n \|\chi_{X_i}f_i\|^p = \sum_{i=1}^n \int_{X_i} |f_i(t)|^p\ \dd m(t)
\leq \int_\Omega \max_{i\in\naturals_n} |f_i(t)|^p \ \dd m(t)\,, 
\]
 and so 
$\norm{(f_1,\ldots,f_n)}^{[p]}_n\le \norm{\sum_{i=1}^n \delta_i\otimes f_i}_F$.

Thus $\norm{(f_1,\ldots,f_n)}^{[p]}_n = \norm{\sum_{i=1}^n \delta_i\otimes f_i}_F$, and so the result follows.
\end{proof}

From the above theorem, it follows that, for every $f_1,\ldots,f_n\in \Lspace^p(\Omega)$, we have
\[
\norm{(f_1,\ldots, f_n)}_n^{[p]}=\norm{\abs{f_1}\vee\cdots \vee \abs{f_n}}=\norm{(f_1,\ldots, f_n)}_n^{L}\,.
\]
We do not have a similar description of the standard $q$-multi-norm on $\Lspace^p(\Omega)$ when $q>p$.

When $p=1$, it is well-known that $\Lone(\Omega) \projectivetensor E = \Lone(\Omega,E)$ for any
Banach space $E$, and so the standard $1$-multi-norm on $\Lone(\Omega)$ is the
maximum multi-norm ({\it cf.} \cite[Theorem 4.23]{DP08}). Thus, for $f_1,\ldots,f_n\in\Lone(\Omega)$, we have
\begin{align}\label{max, [1,1], lattice multi-norm}
	\norm{(f_1,\ldots, f_n)}_n^{\max}=\norm{\abs{f_1}\vee\cdots\vee\abs{f_n}}=\norm{(f_1,\ldots, f_n)}_n^{[1]}=\norm{(f_1,\ldots, f_n)}_n^{L}\,.
\end{align}

\section{The extension of the standard $q$-multi-norm}
\label{The extension of the standard q-multi-norm}

\noindent In this section, we shall give another description of the $(p,q)$-multi-norm based on a space $E$; the description will be required for our main theorem in
  \S\ref{Multi-norm and homology}. In that later section, we shall need to prove the amenability of a locally compact group $G$ by using information about
  $\operators(\Lone(G),\Lspace^p(G))$; the latter information is provided directly by the injectivity of $\Lspace^p(G)$. The main result of this section will
 give us a necessary bridge between $G$ and $\operators(\Lone(G),\Lspace^p(G))$.

Let $E$ be a Banach space, and let $((F^n,\norm{\cdot}_n):\ n\in\naturals)$ be a multi-normed space with $F\neq \set{0}$. For each $n \in \naturals$, we define a norm $\norm{\cdot}_n^{F}$ on the space $E^n$ by setting
\[
\norm{(x_1,\ldots,x_n)}^{F}_n=\sup\set{\norm{(Tx_1,\ldots,Tx_n)}_n\colon T\in \operators(E, F)_{[1]}}\quad (x_1,\ldots, x_n \in E)\,.
\]
It is immediately checked that $(\norm{\cdot}^F_n:\ n\in\naturals)$ is a multi-norm based on $E$ and that 
\begin{equation*}\label{eqI}
\mboperators(E, F)=\operators(E, F)\quad {\rm with}\quad \norm{T}_{mb}=\norm{T}\quad (T \in \mboperators(E, F))\,
\end{equation*}
when $\mboperators(E,F)$ is calculated with respect to the multi-norm $(\norm{\cdot}^F_n:\ n\in\naturals)$ based on $E$. 

Now suppose that $(\tnorm{\,\cdot\,}_n: n \in \naturals)$ is a multi-norm based on $E$ with the property that, with respect to this new multi-norm,  we have
\begin{align}\label{Eq: The extension of the standard $q$-multi-norm 1}
	\mboperators(E, F)=\operators(E, F)\quad {\rm with}\quad \norm{T}_{mb}=\norm{T}\quad\textrm{for each}\ T \in \mboperators(E, F)\,.
\end{align}
Then we see that 
\begin{align*}
	\norm{\tuple{x}}^{F}_n&=\sup\set{\norm{(Tx_1,\ldots,Tx_n)}_n\colon T\in \operators(E, F)_{[1]}}\\
	&=\sup\set{\norm{(Tx_1,\ldots,Tx_n)}_n\colon T\in \mboperators(E, F)_{[1]}}\leq \tnorm{\tuple{x}}_n\,
\end{align*}
for each $\tuple{x}=(x_1,\ldots, x_n) \in E^n$ and $n \in \naturals$. Thus  $(\norm{\cdot}^F_n:\ n\in\naturals)$ is the minimum multi-norm in $\E_E$ satisfying condition
 \eqref{Eq: The extension of the standard $q$-multi-norm 1}.

\begin{definition}
The multi-norm $(\norm{\cdot}_n^{F}: n \in \naturals)$ described above is the \emph{extension} to $E$ of the multi-norm on $F$.
\end{definition}

There is a discussion of extensions of multi-norms in \cite[$\S 6.5$]{DP08}. 

Now, let $(\Omega,\mu)$ be a measure space, and take $p,q$ with $1\le p\leq q < \infty$. For the rest of this section, we shall suppose also that $\Lspace^p(\Omega)$ is \emph{infinite--dimensional}. This is the same as requiring that, for every $n\in\naturals$, there exist pairwise--disjoint, measurable subsets $X_1,\ldots, X_n$  of $\Omega$ such that $0<\mu(X_i)<\infty$ for all $i\in\naturals_n$.

Again, we write $p'$ for the conjugate index of $p$, and set $F=\Lspace^{p'}(\Omega)$. For each $n\in\naturals$, 
let $D_n$ be the set of elements $(\lambda_1,\ldots,\lambda_n) \in (F_{[1]})^n$ such that the subsets $\support \lambda_1,\ldots, \support \lambda_n$ of $\Omega$ are pairwise disjoint. Then it is immediate from the definition of the standard $q$-multi-norm on $\Lspace^p(\Omega)$ that
\begin{align}\label{standard multi-norm, different formula}
	\norm{(f_1,\ldots,f_n)}_n^{[q]}=\sup\set{\left(\sum_{i=1}^{n}\abs{\duality{f_i}{\lambda_i}}^{q}\right)^{1/q}\colon (\lambda_1,\ldots,\lambda_n)\in D_n}
\end{align}
for every $(f_1,\ldots,f_n)\in\Lspace^p(\Omega)^n$ and every $n\in\naturals$.

For each $n\in\naturals$, set
\[
B_n(\dual{E})=\set{ (\dual{T}\varphi_1,\ldots, \dual{T}\varphi_n)\colon T \in \operators(E, \Lspace^{p}(\Omega))_{[1]},\, (\varphi_1,\ldots, \varphi_n) \in  D_n}\subset (\dual{E})^n\,.
\]

\begin{lemma}\label{4.6.19}
Let $E$ be a Banach space. Then $B_n(\dual{E})=\set{ \tuple{\lambda} \in (\dual{E})^n: \mu_{{p}, n}(\tuple{\lambda})\leq 1}$ for each $n\in\naturals$.
\end{lemma}
\begin{proof}
Set $C_n(\dual{E})=\set{ \tuple{\lambda} \in (\dual{E})^n: \mu_{p, n}(\tuple{\lambda})\leq 1}$.

Let $T \in \operators(E, \Lspace^{p}(\Omega))_{[1]}$ and ($\varphi_1,\ldots, \varphi_n) \in  D_n$, so that $T': F\to E'$. For each $i \in \naturals_n$, set  $X_i=\support \varphi_i$,  and then set 
$ \tuple{\lambda}=(\dual{T}\varphi_1,\ldots, \dual{T}\varphi_n) \in (\dual{E})^{n}$.
For each $x \in E_{[1]}$, we have
\begin{align*}
\lp\sum_{i=1}^n\abs{\duality{\dual{T}\varphi_i}{x}}^{p}\rp^{1/p}&=\lp\sum_{i=1}^n\abs{\duality{\varphi_i}{Tx}}^{p}\rp^{1/p}\leq \lp\sum_{i=1}^n\norm{\chi_{X_i}Tx}
^{p}\rp^{1/p} \leq \norm{Tx}
\leq 1\,.
\end{align*}
Hence $\mu_{p,n}(\tuple{\lambda})\leq 1$, and so $B_n(\dual{E})\subset C_n(\dual{E})$.

Conversely, let $\tuple{\lambda}=(\lambda_1,\ldots,\lambda_n) \in C_n(\dual{E})$, and then choose pairwise--disjoint, measurable subsets $X_1,\ldots, X_n$ of $\Omega$ with $0<\mu(X_i)<\infty\,\; (i \in \naturals_n)$.  
Set 
\begin{align*}
	\varphi_i=\begin{cases}
	\dfrac{\chi_{X_i}}{\mu  (X_i)^{1/p'}}&\ \textrm{when}\ p>1\\
	\chi_{X_i}&\ \textrm{when}\ p=1
	\end{cases}\quad (i \in \naturals_n)\,,
\end{align*}
so that $(\varphi_1,\ldots,\varphi_n) \in D_n$. Next set
\[
T=\sum_{i=1}^{n}\lambda_i\otimes\frac{\chi_{X_i}}{\mu(X_i)^{1/p}}\in E'\otimes L^p(\Omega) \subset \operators(E,\Lspace^{p}(\Omega))\,,
\]
where we again use the identification of \eqref{tensor form of finite-rank}. For $x\in E$, we have
\begin{align*}
\norm{Tx}=\flexiblenorm{\sum_{i=1}^n \duality{x}{\lambda_i}\frac{\chi_{X_i}}{\mu(X_i)^{1/p}}}
= \left(\sum_{i=1}^n\abs{\duality{x}{\lambda_i}}^{p}\right)^{1/p} \leq\mu_{p,n}(\tuple{\lambda})\norm{x}\le \norm{x}\,.
\end{align*}
It follows that $T \in \operators(E,\Lspace^{p}(\Omega))_{[1]}$. Since it can be seen that $\tuple{\lambda}=(\dual{T}\varphi_1,\ldots, \dual{T}\varphi_n)$, we have 
$C_n(\dual{E})\subset B_n(\dual{E})$.

Hence $B_n(\dual{E})=C_n(\dual{E})$, as required. 
\end{proof}

The following proposition now follows from Lemma \ref{4.6.19} and equation \eqref{standard multi-norm, different formula}.

\begin{proposition}\label{4.6.20}
Let $E$ be a Banach space, and take $p,q$ with $1\le p\leq q < \infty$. Let $\Omega$ be a measure space such that $\Lspace^{p}(\Omega)$ is infinite--dimensional.
 Then the extension to $E$ of the standard $q$-multi-norm on $\Lspace^{p}(\Omega)$ is the $(p,q)$-multi-norm on $E$.\enproof
\end{proposition}

When $p>1$, since $\bidual{\Lspace^p(\Omega)}=\Lspace^p(\Omega)$, we can do the same as the above for $\bidual{E}$.

\begin{proposition}\label{4.6.21}
Let $E$ be a Banach space, and take $p,q$ with $1< p\leq q < \infty$. Let $\Omega$ be a measure space such that $\Lspace^{p}(\Omega)$ is infinite--dimensional. 
 Then, for each $\Phi_1,\ldots,\Phi_n \in \bidual{E}$, we have
\[
\norm{(\Phi_1,\ldots,\Phi_n)}_n^{(p,q)}=\sup \norm{(\bidual{T}(\Phi_1),\ldots,\bidual{T}(\Phi_n))}_n^{[q]}
\]
where the supremum is taken over all $T \in \operators(E, \Lspace^{p}(\Omega))_{[1]}$.
\end{proposition}
\begin{proof}
Let $\tuple{\Phi}=(\Phi_1,\ldots,\Phi_n)\in (\bidual{E})^n$. By Lemma \ref{weak (p,q)-multi-norm on dual space}, we have
\[
\norm{\tuple{\Phi}}_n^{(p,q)}=\sup\set{\lp\sum_{i=1}^n\abs{\duality{\Phi_i}{\lambda_i}}^{q}\rp^{1/q}: \tuple{\lambda} \in (\dual{E})^n,\, \mu_{p,n}(\tuple{\lambda})\leq 1}\,.
\]
By Lemma \ref{4.6.19}, this is equal to
$$
\norm{\Phi}_n^{(p,q)}=\sup\set{\lp\sum_{i=1}^n\abs{\duality{\Phi_i}{\dual{T}\varphi_i}}^{q}\rp^{1/q}: T \in \operators(E, \Lspace^{p}(\Omega))_{[1]},\,  (\varphi_1,\ldots,\varphi_n)\in  D_n} \,.
$$
Hence
\begin{align*}
\norm{\Phi}_n^{(p,q)} &=\sup\set{\lp\sum_{i=1}^n\abs{\duality{\bidual{T}\Phi_i}{\varphi_i}}^{q}\rp^{1/q}: T \in \operators(E, \Lspace^{p}(\Omega))_{[1]},\, (\varphi_1,\ldots,\varphi_n) \in  D_n} \\
&=\sup \set{\norm{(\bidual{T}(\Phi_1),\ldots,\bidual{T}(\Phi_n))}_n^{[q]}: T \in \operators(E, \Lspace^{p}(\Omega))_{[1]}}\,
\end{align*}
by equation \eqref{standard multi-norm, different formula}, which gives the result.
\end{proof}

\section{Left $(p,q)$-multi-invariant means} 
\noindent In this section, we shall generalize the concept of a left-invariant mean for a locally compact group. 

\begin{definition}
Let $G$ be a locally compact group, and take $p,q$ with $1\leq p\leq q<\infty$. A functional $\Lambda \in \dual{\Linfty(G)}$ is \emph{left $(p,q)$-multi-invariant} if the set
 $\set{s\cdot \Lambda\colon s \in G}$ is multi-bounded in the $(p,q)$-multi-norm. 
The group $G$ is {\it left $(p,q)$-amenable} if there exists a left $(p,q)$-multi-invariant mean on $\Linfty(G)$.
\end{definition}

The idea behind this definition is to attempt to measure the `left-invariance' of a mean $\Lambda \in \dual{\Linfty(G)}$ by measuring the growth of the sets $\set{s\cdot \Lambda: s \in F}$ 
as $F$ ranges through all the finite subsets of $G$. 

It follows immediately from the multi-norm axiom (A4) and Theorem \ref{equivalence of (p,q)-multi-norm on Lone} that we have the following implications for a mean $\Lambda\in\dual{\Linfty(G)}$:
for every $1\le p<q<r<\infty$, we have
\begin{align*}
\textrm{left-invariant}\Rightarrow \textrm{left}\ (q,q)\text{-invariant} \Rightarrow &\textrm{left}\ (p,q)\text{-invariant}\\
&\quad\quad\Updownarrow\\
 &\textrm{left}\ (1,q)\text{-invariant}\Rightarrow \textrm{left}\ (r,r)\text{-invariant}\,.
\end{align*}

 We shall now show that  the left $(p,q)$-multi-invariance property is preserved when passing from a functional on $L^{\infty}(G)$ to an appropriate mean; this is analogous to a
standard property of left-invariant functionals.

\begin{lemma}\label{(p,q)-multi-invariance and absolute value}
Let $G$ be a locally compact group, and take $p,q$ with $1\leq p\leq q<\infty$. Suppose that $\Lambda$ is a non-zero, left $(p,q)$-multi-invariant functional
 on $\Linfty(G)$. Then $\abs{\Lambda}/\norm{\Lambda}$ is a left $(p,q)$-invariant mean on $\Linfty(G)$, and $G$ is left $(p,q)$-amenable.
\end{lemma}

\begin{proof}  Recall that $\dual{\Linfty(G)}$ can be identified isometrically as a Banach lattice with  $\Lone(\Omega)$ for some measure space $(\Omega,\mu)$. 
Set $\widetilde{\Lambda}:=\abs{\Lambda}/\norm{\Lambda}$. Since $\norm{\abs{\Lambda}}=\norm{\Lambda}=\langle 1,\,\abs{\Lambda}\rangle$, it is clear that  $\widetilde{\Lambda}$ is  a mean on $\Linfty(G)$. 
Since $\mu_{p,n}(\varphi_1,\ldots,\varphi_n)=\mu_{p,n}(\psi_1,\ldots,\psi_n)$ for every $n\in\naturals$ and every $\varphi_1,\ldots,\varphi_n,\psi_1,\ldots,\psi_n\in\Linfty(\Omega)$ 
with $\abs{\varphi_i}=\abs{\psi_i}$ ($i\in\naturals_n$) \cite[2.6]{Jameson}, we see that 
\[
	\norm{(\Lambda_1,\ldots,\Lambda_n)}^{(p,q)}_n=\norm{(\abs{\Lambda_1},\ldots,\abs{\Lambda_n})}^{(p,q)}_n\qquad(\Lambda_1,\ldots,\Lambda_n\in\dual{\Linfty(G)})\,.
\]
Now note that $\abs{s\cdot\Lambda}=s\cdot\abs{\Lambda}$ for every $s\in G$, and so  
 $\{s\cdot\widetilde{\Lambda}\colon s \in G\}$ is multi-bounded in the $(p,q)$-multi-norm.   The result follows.
\end{proof}

It turns out that left $(p,q)$-amenability is the same as amenability for a locally compact group $G$, as Theorem \ref{(p,q)-amenable is amenable}, given below, will show.  We shall use the Ryll-Nardzewski fixed point theorem; to be explicit, we first quote a special form of the version of this theorem given in \cite[Theorem A.2.2]{Green} and \cite[\S 2.36]{Paterson}.

\begin{theorem}\label{RN}  Let $E$ be a  Banach   space, and let $K$ be a convex, weakly compact subset of $E$.  Suppose that $\Sigma$ is a semigroup of affine maps from $K$ to $K$ such that $\left\Vert Tx -Ty\right\Vert = \left\Vert  x - y\right\Vert$ for all $x,y\in K$ and $T\in \Sigma$.  Then there exists $x_0\in K$ such that $Tx_0=x_0$ for each $T\in \Sigma$.
\enproof
\end{theorem}

\begin{theorem}\label{(p,q)-amenable is amenable}
Let $G$ be a locally compact group, and take $p,q$ with $1\le p\le q<\infty$. Then $G$ is amenable if and only if $G$ is left $(p,q)$-amenable.
\end{theorem}
\begin{proof}
We need to prove only the `if' part.  

So, suppose that $\Lambda$ is a left $(p,q)$-invariant mean on $\Linfty(G)$; that is $\set{s\cdot\Lambda\colon s\in G}$ is $(p,q)$-multi-bounded in $\dual{\Linfty(G)}$. By Corollary \ref{multi-bounded and weak compactness in Lone} and either 
Lemma \ref{absolutely convex hull and multi-bounded} or the Kre\u{\i}n-\u{S}mulian theorem, the closed convex hull $K$ of $\set{s\cdot\Lambda\colon s\in G}$ is weakly compact. For each $s\in G$, consider the map $L_s:\Psi\mapsto s\cdot\Psi\,,\ K\to K$. We obtain a group $\Sigma : = \set{L_s\colon s\in G}$ of isometric affine maps. By Theorem \ref{RN}, there exists $\Lambda_0\in K$ which is a common fixed point for the set $\set{L_s:\ s\in G}$. Obviously, $\Lambda_0$ must be a left-invariant mean on $\Linfty(G)$. Hence the group $G$ is amenable. 
\end{proof}

\begin{remark}
When $\Linfty(G)$ has a left $(1,1)$-multi-invariant mean $\Lambda$, a left-invariant mean on $\Linfty(G)$ can be explicitly constructed, as follows. Consider $\Lambda$ as an element of the real Banach lattice $\dual{\Linfty_\reals(G)}$. For each finite subset $F=\set{s_1,\ldots, s_n}$ of $G$, we set 
\[
	\Psi_F:= (s_1\cdot\Lambda)\vee\cdots \vee(s_n\cdot\Lambda)\,.
\]
Then we have an upward-directed net of positive linear functionals on $\Linfty(G)$. This net is bounded because $\Lambda$ is left $(1,1)$-multi-invariant, where we use equation \eqref{max, [1,1], lattice multi-norm} ({\it cf.} Remark \ref{(1,1) multi-norm is maximum}), and so its weak$^*$ limit $\Psi$ exists and $\Psi$ must be the supremum of $\set{s\cdot \Lambda\colon s\in G}$. It follows that $\Psi$ is left-invariant.
\end{remark}

\begin{remark} There is an obvious definition of a {\it right} $(p,q)$-{\it multi-invariant mean}. Set $A=\Lone(G)$, and let $\Lambda \in \bidual{A}$ be a left $(p,q)$-multi-invariant mean. Define $\theta: A\rightarrow A$ by 
\[
\theta(a)(s)=a(s^{-1})\Delta(s^{-1})\quad (a \in A, s \in G)\,.
\]
Then $\bidual{\theta}:\bidual{A}\rightarrow \bidual{A}$ takes the set $\set{s\cdot \Lambda: s \in G}$ to the set $\set{\bidual{\theta}(\Lambda)\cdot s: s \in G}$, and $\dual{\theta}(1)=1$. Since $\bidual{\theta}$ automatically belongs  to $\mboperators(\bidual{A}, \bidual{A})$, it follows that $\bidual{\theta}(\Lambda)$ is a right $(p,q)$-multi-invariant mean on $G$.
\end{remark}

\section{Injectivity and flatness of the module $\Lspace^p(G)$}

\label{Multi-norm and homology}

\noindent Let $G$ be a locally compact group, and take $p\in(1,\infty)$. In this section, we shall give an answer to the question of when $\Lspace^p(G)$ is injective and when it is flat in $\Lmod$. Now we shall write $\norm{\cdot}_p$ for the norm on $\Lspace^p(G)$; we {take $q$ to be the conjugate index to $p$.} 

First, we shall prove that the injectivity of $\Lspace^p(G)$ in $\Lmod$ implies the amenability of $G$. For this, we shall use a coretraction problem to show that, in the case where $\Lspace^p(G)$ is injective, $\Linfty(G)$ must have a left $(p,p)$-multi-invariant mean, and then we shall apply the result from the previous section. 

We set $J=\mathcal{B}(\Lone(G),\Lspace^p(G))$. We now define an action of $G$ on the space $J$ by 
\begin{align*}
\left(t * U\right)(a)=t\cdot U(t^{-1} \cdot a)\quad (a \in \Lone(G))\,
\end{align*}
for each $U\in J$ and $t\in G$. For each $U \in J$ and $a \in \Lone(G)$, the map $t\mapsto (t*U)(a),\,\; G\rightarrow \Lspace^p(G),$ is continuous; this follows from the inequality
\begin{align*}
\left\|t\cdot U(t^{-1}\cdot a)-U(a)\right\|_{p}&\leq\left\|t\cdot U(t^{-1}\cdot a)-t\cdot U(a)\right\|_{p}+\left\|t\cdot U(a)-U(a)\right\|_{p}\\
                         &=\left\|U(t^{-1}\cdot a-a)\right\|_{p}+\left\|t\cdot U(a)-U(a)\right\|_{p} \\
                         &\leq \norm{U}\norm{t^{-1}\cdot a-a}_{1}+\left\|t\cdot U(a)-U(a)\right\|_{p}
\end{align*}
and the continuity of translation in $\Lone(G)$ and $\Lspace^p(G)$ \cite[Proposition 3.3.11]{HGD}.

\begin{proposition}
There is a Banach left $\Lone(G)$-module structure on $J$ given by a product $*$\,, where 
\begin{align}\label{new module action on J}
\left(b * U\right)(a)&=\int_{G}b(t)\left(t * U\right)(a) \dd m(t) \quad (a, b \in \Lone(G),\, U \in J)\,.
\end{align}
\end{proposition}
\begin{proof}
This is similar to the proof that $\Lspace^{p}(G)$ is a left $\Lone(G)$-module \cite[Theorem 3.3.19]{HGD}. 

Fix $U \in J$ and $a, b \in \Lone(G)$, and let $\psi \in C_{00}(G)$. By H\"{o}lder's inequality, we have
\[
\int_{G}\left|U(t^{-1}\cdot a)(t^{-1}s)\right|\left|\psi(s)\right| \dd m(s)\leq \norm{t\cdot U(t^{-1}\cdot a)}_{p}\norm{\psi}_{q}\leq \norm{U}\norm{a}_{1}\norm{\psi}_{q}
\]
for each $t \in G$. Now define $\Lambda(\psi)$ for $\psi\in \C_{00}(G)$ by
\begin{align*}
\Lambda(\psi)&=\int_{G}(b * U)(a)(s)\psi(s) \dd m(s) \\
&=\int_{G}\left(\int_{G}b(t)U(t^{-1}\cdot a)(t^{-1}s) \dd m(t)\right)\psi(s) \dd m(s)\\
                         &=\int_{G}b(t)\left(\int_{G}U(t^{-1}\cdot a)(t^{-1}s)\psi(s) \dd m(s)\right) \dd m(t)\,.
\end{align*}
Then 
\[
	\left|\Lambda(\psi)\right|\leq \norm{b}_{1}\norm{U}\norm{a}_{1}\norm{\psi}_{q}\,,
\]
and so $\Lambda$ extends to an element of $\dual{\Lspace^{q}(G)}$ of norm at most $\norm{b}_{1}\norm{U}\norm{a}_{1}$. Hence, by the identification of $\dual{\Lspace^{q}(G)}$ with $\Lspace^{p}(G)$, we see that $(b*U)(a) \in \Lspace^{p}(G)$ with
$$
\left\Vert (b*U)(a) \right\Vert \leq \norm{b}_{1}\norm{U}\norm{a}_{1}\,,
$$
and so $b*U \in J$ with $\norm{b* U}\leq \norm{b}_{1}\norm{U}$.

The associativity formula $a * (b* U)=(a*b)*U$ holds for $a,b\in C_{00}(G)$ and $U\in J$, and so it holds for all $a,b\in \Lone(G)$ because $C_{00}(G)$ is dense in $\Lone(G)$.
\end{proof}

We shall denote the above left $\Lone(G)$-module by $\widetilde{J}=(J, \,*\,)$. (We could similarly define a right multiplication on  $J$ such that $J$ becomes a Banach $\Lone(G)$-bimodule.)

Now we define an embedding $\widetilde{\Pi} :\Lspace^p(G) \rightarrow \widetilde{J}$ by
\[
(\widetilde{\Pi}x)(a)=\varphi_{G}(a)x\quad (a \in \Lone(G))\,,
\]
where $x\in\Lspace^p(G)$  and $\varphi_G$ is the augmentation character on $\Lone(G)$. 
 Certainly $\widetilde{\Pi} \in{\operators}(\Lspace^p(G),\widetilde{J})$. For each $b \in \Lone(G)$, we have
\[
\left(b * \widetilde{\Pi}x\right)(a)=\int_G b(t)\varphi_G(t^{-1}\cdot a)\,t\cdot x \dd m(t)=\varphi_G(a)b\star x=\widetilde{\Pi}(b\star x)(a)\quad (a \in \Lone(G))\,,
\]
and so $\widetilde{\Pi}$ is a left $\Lone(G)$-module morphism; further, $\widetilde{\Pi}$ is admissible (a left inverse of $\widetilde{\Pi}$ 
in the category of Banach spaces is the map $U \mapsto U(a_{0})$ for any $a_{0} \in \Lone(G)$ with $\varphi_{G}(a_{0})=1$).

\begin{proposition}\label{a fancy coretract problem} 
Let $G$ be a locally compact group, and take $p\in (1,\infty)$. Suppose that $\Lspace^{p}(G)$ is injective in $\Lmod$. Then the morphism 
$\widetilde{\Pi}$ is a coretraction in $\Lmod$. 
\end{proposition}
\begin{proof}
This follows immediately from the definition of injectivity.
\end{proof} 

The converse of the above proposition is also true, as can be seen from the proof of Theorem \ref{injectivity and amenability}, below.

We shall need the following generalization of \cite[Lemma 5.2]{DP}. 

In the next three results, we suppose that   $\Omega$ is a measure space, and take $p\in (1,\infty)$; the norm on $L^p(\Omega)$ is $\norm{\cdot}_p$.

For $n \in \naturals$, we set $D_n=\set{-1, 1}^n$,
and, for $j \in \naturals_n$, we set
\[
D^+_n(j)=\set{(d_{1},\ldots, d_{n}) \in D_n: d_j=1},\quad D^-_n(j)=\set{(d_{1},\ldots, d_{n}) \in D_n: d_j=-1}\,.
\]

\begin{lemma}\label{5.2.3} 
Let $n\in \naturals$,  and suppose that $F: \naturals_{n}\times \naturals_{n}\rightarrow \Lspace^p(\Omega)$ is a function. Set 
\[
C= {\max}\set{\left(\sum_{j=1}^{n}\flexiblenorm{\sum_{i=1}^{n}d_{i}F(i, j)}_p^{p}\right)^{1/p}:(d_1,\dots,d_n)\in D_n}\,.
\]
Then $\sum_{j=1}^{n}\norm{F(j,j)}_p^p\leq C^p$.
\end{lemma}
\begin{proof} 
Let $d=(d_1,\dots,d_n)\in D_n$, and set $x_{j,d}=\sum^n_{i=1}d_i F(i,j)\;\,(j\in \naturals_n)$. By hypothesis, we have $\sum^n_{j=1}\norm{x_{j,d}}_p^p\leq C^{p}$.
Since there are $2^n$ elements in $D_n$, we have
\[
\sum^n_{j=1}\sum_{d\in D_n}\norm{x_{j,d}}_p^p \leq 2^nC^{p}\,.
\]
For each $j \in \naturals_n$, write $M_j$ for the set of the maps from $\naturals_n\setminus\set{j}$ to $\set{-1,1}$. Then we can write the term 
$ \sum_{d\in D_n}\norm{x_{j,d}}_p^p$ as
\begin{align*}
\sum_{d\in D_n}\norm{x_{j,d}}_p^p &=\sum_{d \in D^+_{n}(j)}\norm{x_{j,d}}_p^{p}+\sum_{d \in D^-_{n}(j)}\norm{x_{j,d}}_p^{p}  \\
           &= \sum_{d \in M_j}\left(\flexiblenorm{\sum_{i\neq j}d_{i}F(i,j)+F(j,j)}_p^{p}+ \flexiblenorm{\sum_{i\neq j}d_{i}F(i,j)-F(j,j)}_p^{p}\right)\\
           &\geq \sum_{d \in M_j}2\norm{F(j,j)}_p^{p} 
           =2^{n}\norm{F(j,j)}_p^{p}\,;
\end{align*}
here, we are using the fact that the function $t\mapsto t^p$ is increasing and convex on $\reals^+$. 
This holds for each $j \in \naturals_n$, and so, summing over $j$, we see that
\[
2^n\sum^n_{j=1}\norm{F(j,j)}_p^p\leq\sum^n_{j=1}\sum_{d\in D_n}\norm{x_{j,d}}_p^p\leq 2^nC^p\,.
\]
Hence we have $\sum^n_{j=1}\norm{F(j,j)}_p^p\leq C^p$, and the result follows.
\end{proof}

For a measurable subset $V$ of $\Omega$  and $U \in \operators(\Lone(\Omega), \Lspace^{\,p}(\Omega))$, we define 
$\chi_{V}U \in \operators(\Lone(\Omega), \Lspace^{\,p}(\Omega))$ by the formula 
\[
(\chi_{V}U)(a)(s)=\chi_{V}(s)U(a)(s)\quad (a \in \Lone(G),\, s \in G)\,.
\]

\begin{proposition}\label{5.3.4} 
 Let  $\set{X_{1},\ldots, X_n}$ and $\set{Y_1,\ldots, Y_n}$ be measurable partitions of $\Omega$. Then, for each bounded linear operator  $R:\operators(\Lone(\Omega), \Lspace^{\,p}(\Omega))\rightarrow \Lspace^{\, p}(\Omega)$, we have
\[
\left(\sum_{i=1}^{n}\norm{\chi_{X_{i}}R(\chi_{Y_{i}}U)}^{p}_{p}\right)^{1/p} \leq \norm{R}\norm{U}\qquad (U \in \operators(\Lone(\Omega), \Lspace^{\,p}(\Omega)))\,.
\]
\end{proposition}
\begin{proof}
Take $U\in\operators(\Lone(\Omega),\Lspace^p(\Omega))$, and define $F:\naturals_{n}\times \naturals_{n}\rightarrow \Lspace^{\, p}(\Omega)$ by
\[
F(i, j)=\chi_{X_{j}}R(\chi_{Y_{i}}U)\quad (i, j \in \naturals_{n})\,.
\]
For each $(d_{1},\ldots, d_n) \in D_n$, we have
\begin{align*}
\sum_{j=1}^{n}\flexiblenorm{\sum_{i=1}^{n}d_{i}F(i,j)}^{p}_p&=\sum_{j=1}^{n}\flexiblenorm{\sum_{i=1}^{n}d_{i}\chi_{X_{j}}R(\chi_{Y_{i}}U)}^{p}_p=
\sum_{j=1}^{n}\flexiblenorm{\chi_{X_{j}}R\left(\sum_{i=1}^{n}d_{i}\chi_{Y_{i}}U\right)}^{p}_p \\
&=\flexiblenorm{R\left(\sum_{i=1}^{n}d_{i}\chi_{Y_{i}}U\right)}^{p}_p\leq \flexiblenorm{R}^{p}\flexiblenorm{\sum_{i=1}^{n}d_{i}\chi_{Y_{i}}U}^{p}= \norm{R}^{p}\norm{U}^{p}\,,
\end{align*}
and so, by Lemma \ref{5.2.3}, we have
\begin{equation*}
\left(\sum_{j=1}^{n}\norm{F(j, j)}^{p}_p\right)^{1/p}=\left(\sum_{j=1}^{n}\norm{\chi_{X_{j}}R(\chi_{Y_{j}}U)}^{p}_p\right)^{1/p}\leq \norm{R}\norm{U}\,,
\end{equation*}
which gives the result.
\end{proof}

In the following result, we are regarding $\sum_{i=1}^n\dual{U}(f_i)\otimes x_i$ as a finite--rank operator from $\Lone(\Omega)$ to $\Lspace^p(\Omega)$ by using equation
 \eqref{tensor form of finite-rank}. We recall that $q$ is the conjugate index to $p$.

\begin{lemma}\label{5.3.5} 
Let $U \in \operators(\Lone(\Omega), \Lspace^{\, p}(\Omega))$, let $f_1,\ldots, f_n \in \Lspace^{\, q}(\Omega)$ have pairwise--disjoint supports, and let 
$x_1,\ldots, x_n \in \Lspace^{\, p}(\Omega)$ have pairwise--disjoint supports. Set
\[
T=\sum_{i=1}^n\dual{U}(f_i)\otimes x_i: \Lone(\Omega)\rightarrow \Lspace^{\, p}(\Omega)\,.
\]
Then $T \in \operators(\Lone(\Omega), \Lspace^{\, p}(\Omega))$ and $\norm{T}\leq \norm{U}\max\set{\norm{f_i}_{q}\norm{x_i}_p: i \in \naturals_n}$.
\end{lemma}
\begin{proof}
Set $X_i=\support f_i\,\; (i \in \naturals_n)$ and  $C=\max\set{\norm{f_i}_{q}\norm{x_i}_p: i \in \naturals_n}$. For each $a \in \Lone(\Omega)$, we have
\begin{align*}
\norm{Ta}_p^p&=\flexiblenorm{\sum_{i=1}^n\duality{Ua}{f_i} x_i}_p^p=\sum_{i=1}^n\abs{\duality{ Ua}{f_i}}^p \norm{x_i}_p^p \\
&\leq\sum_{i=1}^n\norm{\chi_{X_i}U(a)}^p_p\norm{f_i}_{q}^p\norm{x_i}_p^p \leq C^p\norm{\chi_{X_1\cup\cdots\cup X_n}U(a)}_p^p\leq C^p\norm{Ua}_p^p\,.
\end{align*}
Therefore $\norm{Ta}_p\leq C\norm{Ua}_p$, and the result follows.
\end{proof}

In the theorem below, we shall use the following identity. For each $x \in \Lspace^{\, p}(G)$, $\lambda \in \Linfty(G)$, and $s \in G$, we have
\begin{equation}\label{eq4.3}
(\lambda\cdot s)\otimes x=s^{-1}*\left[\lambda\otimes (s\cdot x)\right]\,.
\end{equation}

\begin{theorem}\label{injectivity and amenability} 
Let $G$ be a locally compact group, and take $p\in(1,\infty)$. Then $\Lspace^{\, p}(G)$ is injective in $\Lmod$ if and only if $G$ is amenable.
\end{theorem}
\begin{proof}
It is well-known that, if $G$ is amenable, then $\Lspace^{\, p}(G)$ is injective: by Johnson's theorem \cite{Johnson72}, $\Lone(G)$ is an amenable Banach algebra, 
and explicitly by \cite[VII.2.29]{Helemskii86} $\dual{E}$ is injective for each $E\in\modL$. Hence $\Lspace^p(G)$ is injective. Thus we need to consider only the converse.
So  we suppose that $\Lspace^{\, p}(G)$ is injective in $\Lmod$; we may also suppose that $G$ is infinite.

Recall that we are setting  
$J=\operators(\Lone(G), \Lspace^p(G))$ and $q=p'$. By Proposition \ref{a fancy coretract problem}, there is a morphism $R \in {}_{\Lone(G)}\operators(\widetilde{J}, \Lspace^p(G))$
 with $R\circ \widetilde{\Pi}=\id_{\Lspace^p(G)}$. For each compact  subset $V$ of $G$ with $m(V)>0$, we define a linear functional $\Lambda_{V}$ on $\Linfty(G)$ by
\[
\duality{\lambda}{\Lambda_{V}}=\frac{1}{m(V)}\int_{V}(R(\lambda\otimes \chi_{V}))(t) \dd m(t) \quad (\lambda \in \Linfty(G))\,.
\]
For each $\lambda \in \Linfty(G)$, we have
\[
\abs{\duality{\lambda}{\Lambda_{V}}}\leq \norm{R(\lambda\otimes \chi_{V})}_{p}\norm{\chi_{V}/m(V)}_{q}\leq \norm{R}\norm{\lambda}_{\infty}\norm{\chi_{V}}_{p}\norm{\chi_{V}/m(V)}_{q}
=\norm{R}\norm{\lambda}_{\infty}\,,
\]
and so $\Lambda_{V} \in \dual{\Linfty(G)}$ with $\norm{\Lambda_{V}}\leq \norm{R}$. Let $\V$ be the family of compact neighbourhoods of the 
identity $e$ in $G$, and set $V_{1}\leq V_{2}$ if $V_2 \subset V_1$. Then $(\V, \leq)$ is a directed set. Let $\Lambda$ be a weak$^*$ 
accumulation point in $\dual{\Linfty(G)}$ of the bounded net $\set{\Lambda_{V}: V \in \V}$. By passing to a subnet, we may suppose that
 $\set{\Lambda_{V}: V \in \V}$ converges to $\Lambda$.  Clearly $\duality{1}{\Lambda}=1$ since, for each $V \in \V$, we have
\[
\duality{1}{\Lambda_{V}}=\frac{1}{m(V)}\int_{V}(R(\widetilde{\Pi}\chi_{V}))(t) \dd m(t)=\frac{1}{m(V)}\int_{V}\dd m(t)=1\,,
\]
and so $\Lambda$ is non-zero.  We {\it claim} that $\Lambda$ is left $(p,p)$-multi-invariant.

Take $n \in \naturals$, and consider distinct elements $s_1,\ldots, s_n$ of $G$. Choose $V \in \V$ such that the sets $s_{1}V, \ldots,s_{n}V$ are pairwise disjoint.
 Let $U \in J$, and let $\mathbf{X}=\set{X_1,\ldots, X_n}$ be a measurable partition of $G$. Take $f_1,\ldots, f_n \in \Lspace^{\, q}(G)_{[1]}$ with 
$\support f_i \subset X_i\,\; (i \in \naturals_n)$, and finally set
\[
T=\sum_{i=1}^{n} \dual{U}(f_i)\otimes \chi_{s_{i}V}: \Lone(G)\rightarrow \Lspace^p(G)\,.
\]
By Lemma \ref{5.3.5}, $T\in J$ and $\norm{T}\leq \norm{U}m(V)^{1/p}$.

For each $i \in \naturals_n$, we have
\begin{align*}
m(V)\duality{ f_i}{\bidual{U}(s_{i}\cdot \Lambda_{V})} &=m(V)\duality{\dual{U}(f_i)}{s_{i}\cdot \Lambda_{V}}\\
&=\int_{V}R((\dual{U}(f_i)\cdot s_i)\otimes \chi_{V})(t) \dd m(t) \\
&=\int_{V}R( \dual{U}(f_i)\otimes (s_i\cdot \chi_{V}))(s_{i}t) \dd m(t) \\
&=\int_{s_{i}V}R(\dual{U}(f_i)\otimes \chi_{s_{i}V})(t) \dd m(t) \\
&=\int_{s_{i}V}R(\chi_{s_{i}V}T)(t) \dd m(t)\,;
\end{align*}
the third equality holds true by \eqref{eq4.3} and because $R$ is a $\Lmod$ homomorphism from $\widetilde{J}$ into $\Lspace^p(G)$ and $\Lspace^p(G)$ is essential in $\Lmod$. Hence, by H\"{o}lder's inequality, we have
\[
\abs{\duality{\dual{U}(f_i)}{s_{i}\cdot \Lambda_{V}}}\leq\norm{\chi_{s_{i}V}R(\chi_{s_{i}V}T)}_{p}m(V)^{\frac{1}{q}-1}\,.
\]
Then, by Proposition \ref{5.3.4}, we have
\begin{align*}
\left(\sum_{i=1}^{n}\abs{\duality{f_i}{\bidual{U}(s_{i}\cdot \Lambda_{V})}}^{p}\right)^{1/p} &\leq \left(\sum_{i=1}^{n}\norm{\chi_{s_{i}V}R(\chi_{s_{i}V}T)}_{p}^{p}\right)^{1/p}m(V)^{\frac{1}{q}-1} \\
&\leq  
\norm{R}\norm{T}m(V)^{\frac{1}{q}-1}\leq \norm{R}\norm{U}m(V)^{\frac{1}{p}}m(V)^{\frac{1}{q}-1}= \norm{R}\norm{U}\,.
\end{align*}
Therefore
\[
\lp\sum_{i=1}^{n}\abs{\duality{ f_i}{\bidual{U}(s_{i}\cdot \Lambda)}}^{p}\rp^{1/p} =\lim_{V}\left(\sum_{i=1}^{n}\abs{\duality{\dual{U}(f_i)}{s_{i}\cdot \Lambda_{V}}}^{p}\right)^{1/p}\leq \norm{R}\,.
\]
Since this is true for all such families $\set{f_1,\ldots,f_n}$ in $\Lspace^q(G)_{[1]}$, we have
\[
\left(\sum_{i=1}^{n}\norm{\chi_{X_i}\bidual{U}(s_{i}\cdot \Lambda)}_p^{p}\right)^{1/p}\leq \norm{R}\,.
\]
Since this is true for each measurable partition $\mathbf{X}=\set{X_1,\ldots,X_n}$ and each $U \in J_{[1]}$, it follows from Proposition \ref{4.6.21} that
\[
\norm{(s_1\cdot \Lambda,\ldots, s_n\cdot \Lambda)}_n^{(p,p)}\leq \norm{R}\,.
\]
Thus $\Lambda$ is a non-zero, left $(p,p)$-multi-invariant functional on $\Linfty(G)$, and so, by Lemma \ref{(p,q)-multi-invariance and absolute value} and Theorem \ref{(p,q)-amenable is amenable}, the group $G$ is amenable.
\end{proof}

The determination of when $\Lspace^p(G)$ is flat in the category $\Lmod$ is an easy consequence of the previous theorem and the following simple observation.

Let $A$ be a Banach algebra, and suppose that $\theta:A\to A$ is a Banach algebra anti-automorphism. For each Banach left $A$-module $E$, we can define a Banach right $A$-module $E_\theta$ as follows. As a Banach space $E_\theta=E$, and the right $A$-module action on $E_\theta$ is defined as
\[
	x\cdot a:=\theta(a)\cdot x\quad\quad(a\in A, x\in E_\theta)\,.
\]
In fact, it is obvious that every Banach right $A$-module has the form $E_\theta$ for some suitable $E\in\Amod$. By going through the definition directly, we obtain the following.

\begin{lemma}
Let $E$ be a Banach left $A$-module. Then $E$ is injective in $\Amod$ if and only if $E_\theta$ is injective in $\modA$. \enproof
\end{lemma}

Consider again the locally compact group $G$, and take $p\in (1,\infty)$, with conjugate index $q$. We see that the dual right $\Lone(G)$-module action $\cdot$ on $\Lspace^{q}(G)=\dual{(\Lspace^{p}(G),\star)}$ is given by
\[
	(h\cdot f)(t)=\int h(st)f(s)\dd m(s)=(\widetilde{f}\star h)(t)\quad\quad(f\in\Lone(G),\ h\in\Lspace^{q}(G))\,,
\]
where $\widetilde{f}(t)=f(t^{-1})\Delta(t^{-1})$  for $f\in \Lone(G)$, so that  $\widetilde{f} \in \Lone(G)$.   Thus $\dual{(\Lspace^{p}(G),\star)}=(\Lspace^{q}(G),\star)_\theta$ where $\theta: f\mapsto\widetilde{f}$ is an isometric anti-automorphism on $\Lone(G)$.

\begin{theorem}\label{flatness and amenability} 
Let $G$ be a locally compact group, and take $p\in(1,\infty)$. Then $\Lspace^{\, p}(G)$ is flat in $\Lmod$ if and only if $G$ is amenable.
\end{theorem}
\begin{proof}
We know that $(\Lspace^{p}(G), \star)$ is flat in $\Lmod$ if and only if the dual module $\dual{(\Lspace^{p}(G),\star)}=(\Lspace^{q}(G),\star)_\theta$ is injective in $\modL$, and hence if and only if $(\Lspace^{q}(G),\star)$ is injective in $\Lmod$. By the main theorem, this holds if and only if $G$ is amenable.
\end{proof}

In summary, by combining Theorems \ref{(p,q)-amenable is amenable}, \ref{injectivity and amenability}, and \ref{flatness and amenability}, we obtain the following theorem.

\begin{theorem}
Let $G$  be a locally compact group, and take $p\in (1,\infty)$. Then the following conditions are equivalent:

\begin{enumerate}
\item  $G$ is amenable;

\item  $\Lspace^p(G)$ is injective in $\Lmod$;

\item  $\Lspace^p(G)$ is flat in $\Lmod$;

\item   $G$ is left $(p,q)$-amenable for all $q\geq p$;

\item   $G$ is left $(p,q)$-amenable for some $q\geq p$;

\item   $G$ is left $(1,q)$-amenable for all $q\ge 1$;

\item   $G$ is left $(1,q)$-amenable for some $q\ge 1$. \enproof
\end{enumerate}
\end{theorem}

\begin{remark}
We also have that, for each $p\in(1,\infty)$, $\Lspace^p(G)$ is [injective / flat] in the categories [$\modL$ / $\LmodL$] if and only if $G$ is amenable. 
\end{remark}

\begin{remark}
There are natural quantitative versions of projectivity, injectivity, and flatness. These were first explicitly introduced and studied in \cite{White}. Let $A$ be a Banach algebra, and let $E \in \Amod$ be injective. We set 
\[
	\inj(E)=\inf \norm{\rho}\,,
\] 
where the infimum is taken over all right-inverse morphisms $\rho$ to the canonical morphism $\Pi$.

It follows from the previous theorem that, for a locally compact group $G$ and $p\in(1,\infty)$, we have $\inj(\Lspace^p(G))=1$ whenever $\Lspace^p(G)$ is injective in $\Lmod$.

Recently, G. Racher \cite{Racher} has proved (by different methods to us) that a discrete group $G$ is amenable whenever $\lspace^{\,2}(G)$ is injective in $\lmod$ with $\inj(\lspace^{\, 2}(G))=1$.
\end{remark}

\section{Semigroup algebras}
\label{Semigroup algebras}

\noindent Let $S$ be a semigroup, with product denoted by juxtaposition. We recall that $S$ is: \mbox{(i) \emph{left-cancellative}} if the map $L_s:t\mapsto st,\ S\to S,$ is injective for each $s\in S$; (ii) \emph{weakly left-cancellative} if  $\set{u\in S: su=t}$ is finite for each $s,t \in S$; (ii) \emph{uniformly weakly left-cancellative} if 
\[
	\sup_{s,t\in S}\cardinal{\set{u\in S: su=t}}<\infty\,.
\]
Further, $S$ is \emph{right-cancellative} if the map $R_s:t\mapsto ts,\ S\to S,$ is injective for each $s\in S$, and $S$ is \emph{cancellative} if it is both left- and right-cancellative.

Let $S$ be a semigroup. Then the Banach space $(\lone(S),\norm{\cdot}_1)$ is a Banach algebra with respect to a product $\star$ satisfying the condition that $\delta_s\star\delta_t=\delta_{st}$ for each $s,t\in S$. The Banach algebra $(\lone(S),\norm{\cdot}_1,\star)$ is the \emph{semigroup} algebra of $S$; for a discussion of this algebra, see \cite{DLS2010}.

Let $S$ be a semigroup. The action of $S$ on $\lone(S)$ is defined by $s\cdot f:=\delta_s\star f$,  
so that
\[
	(s\cdot f)(t)=\sum_{sr=t} f(r)\quad (t\in S)\,
\]
for $s\in S$ and $f\in\lone(S)$. This action can be extended by duality first to an action of $S$ on the space $\dual{\lone(S)}=\linfty(S)$ and then to an action on $\dual{\linfty(S)}=\bidual{\lone(S)}$. This latter extended action is denoted by 
\[
	(s,\Lambda)\mapsto s\cdot\Lambda,\quad S\times \dual{\linfty(S)}\to\dual{\linfty(S)}\,.
\]
 In the case where $S$ is left-cancellative, the action of each $s\in S$ on $\lone(S)$ is an isometry, and so its extension to an action on $\dual{\linfty(S)}$ is also an isometry.

The following proposition is easily checked. 

\begin{proposition}
Let $S$ be a semigroup, and take $p>1$. Then $\lspace^{\,p}(S)$ is a Banach left $\lone(S)$-module if and only if $S$ is uniformly weakly left-cancellative. \enproof
\end{proposition}

Suppose now that $S$ is a non-empty set, and take $p\ge 1$. Then we set $J=\operators(\lone(S), \lspace^{\,p}(S))$. It is easy to see that $J$ can be identified isometrically with the Banach space
\[
	\lspace^{\,\infty,p}(S):=\set{U:S\times S\to\complexs\colon \norm{U}=\sup_{s\in S}\left(\sum_{t\in S}\abs{U(s,t)}^p\right)^{1/p}<\infty};
\]
the identification is given by $U(s,t)=U(\delta_s)(t)$\;\, ($s,t\in G$).

\begin{lemma}\label{estimation lemma for discrete} 
Let $S$ be a non-empty set, and take $p\ge 1$. Suppose that $R:\lspace^{\,\infty,p}(S)\rightarrow \lspace^{\, p}(S)$ is a bounded linear operator. Further, suppose that $U$ and $U_s$ $(s\in S)$ are in $\lspace^{\,\infty,p}(S)$ and are such that $\abs{U_{s}(r,t)}\le \abs{\delta_{s}(t)U(r,t)}$ ($s,r,t\in S$). Then
\[
\left(\sum_{s \in S}\left|R(U_{s})(s)\right|^{p}\right)^{1/p}\leq \norm{U}\norm{R}\,.
\]
\end{lemma}
\begin{proof}
This follows from (the proof of) Proposition \ref{5.3.4}.
\end{proof}

Now suppose  that $S$ is a uniformly weakly left-cancellative semigroup. Define a left $\lone(S)$-module action on $\lspace^{\,\infty,p}(S)$ by
\[
	(f*U)(s,t):=\sum 
	\set{f(r)U(x,y):\  r,x,y\in S,\ rx=s, ry=t}\quad\quad (f\in\lone(S),\ U\in\lspace^{\,\infty,p}(S))\,.
\]
This induces a new Banach left $\lone(S)$-module action on $J$ similar to the one given in \eqref{new module action on J}; to avoid confusion with the standard module action on $J$, we shall denote this new module by $(\widetilde{J},*)$. We shall freely identify $\lspace^{\,\infty,p}(S)$ with $(\widetilde{J},*)$. Define $\widetilde{\Pi}:\lspace^{\,p}(S)\to \widetilde{J}$ by
\[
	\widetilde{\Pi}(g)(f)=\sum_{s\in S}f(s)g\qquad(f\in\lone(S),\  g\in\lspace^{\,p}(S))\,;
\]
in the identification $\widetilde{J}=\lspace^{\,\infty,p}(S)$, we have $\widetilde{\Pi}(g)(s,t)=g(t)$. It is then obvious that $\widetilde{\Pi}$ is an admissible left $\lone(S)$-module homomorphism. 

Again, the next proposition follows from the definition of injectivity.

\begin{proposition}\label{semigroupCoretract}
Let $S$ be a uniformly weakly left-cancellative semigroup, and take $p\ge 1$. Suppose that $\lspace^{\, p}(S)$ is injective in $\lSmod$. Then the morphism $\widetilde{\Pi}:\lspace^{\,p}(S)\rightarrow \widetilde{J}$ is a coretraction. \enproof
\end{proposition}

We shall also need the following, taken from {\cite[Proposition 3.1]{Ramsdensemigroup}}.

\begin{proposition}\label{module homomorphisms when module is injective}
Let $A$ be a Banach algebra, and let $E$ be injective in $\Amod$. Then, for every $Q\in \operators(A,E)$ and $D\subset A$ such that $Q(ab)=aQ(b)$ for all $a\in A$ and $b\in D$, there exists $x_0\in E$ with $Q(b) = bx_0$ for all $b\in D$. \enproof
\end{proposition}

Let $S$ be a semigroup. Then an element $\Lambda \in \dual{\linfty(S)}$ is a {\it mean} on $\linfty(S)$ if $\duality{1}{\Lambda}=\norm{\Lambda}=1$, and $\Lambda$ is {\it left-invariant}
 if $\set{s\cdot \Lambda:\ s \in S}=\set{\Lambda}$; the semigroup $S$ is \emph{left-amenable} if there exists a left-invariant mean on $\linfty(S)$.  See \cite[p. 16]{Paterson} for more details. 

 \begin{definition}  A functional $\Lambda \in \dual{\linfty(S)}$ is {\it left $(p,q)-$multi-invariant\/} if the set $\set{s\cdot \Lambda:\ s \in S}$ is multi-bounded in the $(p,q)$-multi-norm.
\end{definition}

\begin{theorem}\label{left cancellative and injectivity} 
Let $S$ be a left-cancellative semigroup, and take $p\ge 1$. Suppose that $\lspace^{\, p}(S)$ is injective in $\lSmod$. Then $S$ is left-amenable and has a right identity.
\end{theorem}
\begin{proof}
Set $A=\lone(S)$ and $E=\lspace^{\,p}(S)$, and consider the natural injection $\iota:A\to E$. Then $\iota$ is obviously a left $A$-module homomorphism. By Proposition \ref{module homomorphisms when module is injective}, there exists $g_0\in E$ such that $f=f\star g_0$ for every $f\in A$. Since $S$ is left-cancellative, it follows easily that $S$ has a right identity, say $e$.

By Proposition \ref{semigroupCoretract}, there exists $R \in {}_{A}\operators(\widetilde{J}, E)$ with $R\circ \widetilde{\Pi}=\id_{E}$. We define $\Lambda \in \dual{\linfty(S)}$ by
\[
\duality{\lambda}{\Lambda}=(R(\lambda\otimes \delta_{e}))(e)\quad (\lambda \in \linfty(S))\,,
\]
where $\lambda \otimes \delta_e \in \operators(\lone(S), \lspace^{\,p}(S)) =\widetilde{J}$,
 as in  (\ref{tensor form of finite-rank}).  We have
\[
\duality{1}{\Lambda}=(R(\chi_{S}\otimes \delta_{e}))(e)=(R(\widetilde{\Pi}\delta_{e}))(e)=\delta_{e}(e)=1\,.
\]
We {\it claim} that $\Lambda$ is left $(p,p)$-multi-invariant. 

Take $n \in \naturals$, and consider distinct elements $s_{1},\ldots, s_{n}$ of $S$. Let $U \in \operators(A, E)$, and take $f_1,\ldots, f_n \in \dual{E}_{[1]}$ with pairwise--disjoint supports. Define
\[
T=\sum_{i=1}^{n} \dual{U}(f_i)\otimes\delta_{s_i}:A\rightarrow E\,.
\]
By Lemma \ref{5.3.5}, $T\in J$ and $\norm{T}\leq \norm{U}$.

For each $i \in \naturals_n$, since $S$ is left-cancellative, we have
\begin{align*}
\duality{ f_i}{\bidual{U}(s_{i}\cdot \Lambda)}&=\duality{\dual{U}(f_i)}{s_{i}\cdot \Lambda}=R((\dual{U}(f_i)\cdot s_i)\otimes \delta_{e})(e)\\
&=R[\delta_{s_i}*((\dual{U}(f_i)\cdot s_i)\otimes \delta_{e})](s_i)=R(T_{i})(s_i)\,,
\end{align*}
where we set $T_i=\delta_{s_i}*((\dual{U}(f_i)\cdot s_i)\otimes \delta_{e})$. We see that, for each $r,t\in S$, we have
\begin{align*}
	T_i(r,t)&=\sum\set{((\dual{U}(f_i)\cdot s_i)\otimes \delta_{e})(x,y): \ x,y\in S,\ s_ix=r,\ s_iy=t}\\
	&=\sum\set{\dual{U}(f_i)(s_ix)\delta_{e}(y): \ x,y\in S,\ s_ix=r,\ s_iy=t}\\
	&=\sum\set{\dual{U}(f_i)(r)\delta_{s_i}(t): \ x,y\in S,\ s_ix=r,\ s_iy=t}\\
	&=\sum\set{\delta_{s_i}(t)T(r,t): \ x,y\in S,\ s_ix=r,\ s_iy=t}\,,
\end{align*}
and so $T_i(r,t)$ is either $\delta_{s_i}(t)T(r,t)$ or $0$. Thus $\abs{T_i}\le \abs{\delta_{s_i}T}$; here, we are identifying $J$ with $\lspace^{\,\infty,p}(S)$. 
Hence, by Lemma \ref{estimation lemma for discrete}, we have
\[
\lp\sum_{i=1}^{n}\abs{\duality{ f_i}{\bidual{U}(s_{i}\cdot \Lambda)}}^{p}\rp^{1/p} =\left(\sum_{i=1}^{n}\abs{R(T_i)(s_i)}^{p}\right)^{1/p}\leq \norm{R}\norm{U}\,.
\]
Since this is true for all such collections $\set{f_1,\ldots,f_n}$ in $\dual{E}_{[1]}$, by Proposition \ref{4.6.21} when $p>1$ and the same calculation as 
in the proof of Proposition \ref{4.6.21} when $p=1$, we have 
\[
	\norm{(s_{1}\cdot\Lambda, \ldots, s_{n}\cdot\Lambda)}^{(p,p)}_{n}\leq \norm{R}\,.
\]
Therefore the set $\set{s\cdot\Lambda\colon s\in S}$ is $(p,p)$-multi-bounded.

Note that, since $S$ is left-cancellative, the map $L_s:\Psi\mapsto s\cdot \Psi$ on $\dual{\linfty(S)}$ is isometric (as well as being  positive), and so 
$\abs{s\cdot\Lambda}=s\cdot\abs{\Lambda}$  for each $s\in S$. Thus, essentially as in  Lemma \ref{(p,q)-multi-invariance and absolute value}, we see that 
$\abs{\Lambda}/\norm{\Lambda}$ is a left $(p,p)$-multi-invariant mean on $\linfty(S)$. We can then argue in a similar way to that in the proof of Theorem 
\ref{(p,q)-amenable is amenable} by applying the Ryll-Nardzewski fixed point theorem to the semigroup $\set{L_s:\ s\in S}$ to find a left-invariant mean on $\linfty(S)$. 
\end{proof}

The following theorem in the case where $p=1$ was proved  in \cite[Theorem 4.10]{Ramsdensemigroup}.

\begin{theorem}
Let $S$ be a cancellative semigroup, and take $p\in[1,\infty)$. Then $\lspace^{\, p}(S)$ is injective in $\lSmod$ if and only if $S$ is an amenable group.
\end{theorem}
\begin{proof}
Certainly, $\lspace^{\, p}(S)$ is injective in $\lSmod$ whenever $S$ is an amenable group. 

Suppose that $\lspace^{\, p}(S)$ is injective in $\lSmod$. By Theorem \ref{left cancellative and injectivity}, $S$ is left-amenable and has a right identity.
 It remains to prove that $S$ is a group; the argument is similar to that in \cite[Theorem 4.10]{Ramsdensemigroup}.  Since $S$ is cancellative, a right identity $e$ of $S$
 must be the identity of $S$. For each $t\in S$, we consider a map $Q_t:\lone(S)\to \lspace^{\,p}(S)$ defined as
\[
	Q_t\colon\sum_{s\in S}\alpha_s\delta_s\mapsto\sum_{s\in S}\alpha_{st}\delta_s\,,
\]
so that $Q_t(\delta_{st}) =\delta_s\,\;(s\in S)$. 
Since  $R_t$ is injective on $S$, we have  $Q_t(f\star\delta_{st})=f\star Q_t(\delta_{st})$ for all $f\in\lone(S)$ and $s\in S$.
 By Proposition \ref{module homomorphisms when module is injective}, there exists $a_t\in \lspace^{\,p}(S)$ such that 
\[
Q_t(\delta_{st})=\delta_{st}\,\star\, a_t\quad(s\in S)\,. 
\]
Thus $	\delta_s= 	\delta_{st}\,\star\,a_t\,\;(s\in S)$; in particular, $	\delta_e= 	\delta_{t}\,\star\,a_t$.  This implies that there exists $u\in S$ with $tu=e$.
 Now $utu =ue =eu$, and so the injectivity of $R_u$  implies that $ut=e$. 
Hence $u$ is the inverse of $t$ in $S$.  This shows that  $S$ is an (amenable) group. 
\end{proof}

\begin{corollary}
Let $S$ be a right--cancellative semigroup, and take $p\in [1,\infty)$. Suppose that $\lspace^{\,p}(S)$ is flat in $\lSmod$. 
Then $S$ is right--amenable and has a left identity. If, furthermore, $S$ is cancellative, then $S$ is an amenable group. \enproof
\end{corollary}

\newcommand{\email}{\texttt}

\noindent   H.\ Garth\ Dales, \\
 Department of  Mathematics and Statistics\\
Fylde College\\
University of Lancaster\\
Lancaster LA1 4YF\\
United Kingdom 
\email{g.dales@lancaster.ac.uk}

\medskip

\noindent Matthew\ Daws\\
    Department of Pure Mathematics,\\
University of Leeds, \\ Leeds LS2 9JT,\\  
United Kingdom 
   \email{mdaws@maths.leeds.ac.uk}

\medskip

\noindent Paul Ramsden\\
5 Brookhill Crescent\\
Leeds\\
LS17 8QB\\
United Kingdom
\email{paul@virtualdiagonal.co.uk}

\medskip

\noindent    Hung\ Le\ Pham\\
   School of Mathematics, Statistics, and Operations Research,\\
Victoria University of Wellington, \\ Wellington 6140,\\ New Zealand
   \email{hung.pham@vuw.ac.nz}
     
     \end{document}